\documentclass[11pt,a4,onesided]{amsart}

\usepackage{hyperref}
\usepackage{amsfonts}
\usepackage{amssymb}
\usepackage{amsmath}
\usepackage{amsthm}
\usepackage{latexsym}
\usepackage{color}
\usepackage{amscd}
\usepackage[all,cmtip]{xy}
\usepackage{dsfont}
\usepackage{mathrsfs}
\usepackage{graphicx}
\usepackage{enumitem} 
\usepackage[margin=1.6in]{geometry}
\usepackage{esint}
\usepackage{epsfig}

\allowdisplaybreaks

\numberwithin{equation}{section}

\newtheorem{thm}{Theorem}[section]
\newtheorem{cor}[thm]{Corollary}
\newtheorem{lem}[thm]{Lemma}
\newtheorem{prop}[thm]{Proposition}
\newtheorem{defn}[thm]{Definition}
\newtheorem{rem}[thm]{Remark}

\newcommand\R{{\mathbb R}}

\def\ts{\stackrel{2-drift}{\relbar\joinrel\relbar\joinrel\relbar\joinrel\relbar\joinrel\rightharpoonup}}
\def\tss{\stackrel{2s-drift}{\relbar\joinrel\relbar\joinrel\relbar\joinrel\relbar\joinrel\rightharpoonup}}

\def\div{{\rm div}}
\def\dsp{\displaystyle}

\let\oldmarginpar\marginpar
\renewcommand\marginpar[1]{\-\oldmarginpar[\raggedleft\footnotesize #1]%
{\raggedright\footnotesize #1}}

\newcommand{\ve}{\varepsilon}
\newcommand{\iti}{\int_0^T}
\newcommand{\itt}{\int_0^t}
\newcommand{\iy}{\int_{Y^0}}
\newcommand{\ip}{\int_{\partial\Sigma^0}}
\newcommand{\id}{\int_{\Omega_{\varepsilon}}}
\newcommand{\ib}{\int_{\partial\Omega_{\varepsilon}}}

\newcommand{\nx}{\nabla_x}
\newcommand{\ny}{\nabla_y}

\newcommand{\itr}{\int_{\mathbb{R}^d}}

\newcommand{\drfyg}{\left(t,x-\frac{b^* t}{\varepsilon},\frac{x}{\varepsilon}\right)}
 
\newcommand{\drfg}{\left(t,x-\frac{b^* t}{\ve}\right)}
\newcommand{\drift}[1]{\widehat{#1}}
\newcommand{\mdrift}[1]{\check{#1}}

\def\div{{\rm div}}
\def\dsp{\displaystyle}
\def\ts{\stackrel{2-drift}{\relbar\joinrel\relbar\joinrel\relbar\joinrel\relbar\joinrel\rightharpoonup}}
\def\tss{\stackrel{2s-drift}{\relbar\joinrel\relbar\joinrel\relbar\joinrel\relbar\joinrel\relbar\joinrel\rightharpoonup}}
\def\R{\mathbb{R}}

\def\signga{\bigskip {\small \sc
      CMAP, UMR CRNS 7641, \'Ecole Polytechnique,
      Route de Saclay, Palaiseau F91128, France \par E-mail:}
    \tt{\small gregoire.allaire@polytechnique.fr} }

\def\signhh{\bigskip {\small \sc
      DPMMS, CMS, University of Cambridge,
      Wilberforce road, Cambridge CB3 0WB, UK
      \par E-mail:}
    \tt{\small H.Hutridurga@dpmms.cam.ac.uk} }
    
\begin{document}

\title[Homogenization of Nonlinear Adsorption]{Upscaling nonlinear adsorption in periodic porous media - Homogenization approach.}

\author{Gr\'egoire Allaire, Harsha Hutridurga}

\begin{abstract}
We consider the homogenization of a model of reactive flows through periodic porous media 
involving a single solute which can be absorbed and desorbed on the pore boundaries. This 
is a system of two convection-diffusion equations, one in the bulk and one on the pore 
boundaries, coupled by an exchange reaction term. The novelty of our work is to consider 
a nonlinear reaction term, a so-called Langmuir isotherm, in an asymptotic regime of 
strong convection. We therefore generalize previous works on a similar linear model 
\cite{AllaireHutridurga, Allaire10, AllaireMikelic10}. 
Under a technical assumption of equal drift velocities in the bulk and on the pore boundaries, 
we obtain a nonlinear monotone diffusion equation as the homogenized model. Our main 
technical tool is the method of two-scale convergence with drift \cite{Marusic05}. 
We provide some numerical test cases in two space dimensions to support our theoretical 
analysis. 
\end{abstract}

\maketitle

\graphicspath{{figures/}}

\DeclareGraphicsExtensions{.pdf,.png,.jpg}


\section{Introduction and setting of the problem}
Solute transport in porous media is a topic of interest for chemists, geologists and
environmental scientists. The phenomena that affect solute transport 
are convection, diffusion and the chemical reactions that the solutes might undergo. Since 
the seminal work of G.I.~Taylor \cite{Taylor53}, dispersion phenomenon (i.e., the phenomenon of the spreading of solutes in a fluid medium) has attracted 
a lot of attention. Mathematical modeling of solute transport through porous media 
can be approached via various means. One possibility is to describe the physical 
and chemical phenomena at the pore (microscopic) scale and then perform an `upscaling' 
or `averaging' in order to derive a macroscopic model. The theory of Homogenization 
(see e.g. \cite{HOR, Jikov94}) is a mathematically rigorous approach for averaging 
partial differential equations and carrying out the above program. Upscaling techniques, 
homogenization being one of them, are necessary to perform numerical simulations at a 
reasonable computational cost since it is very difficult, if not impossible, to perform numerical simulations of pore scale models.

Many works have been devoted to the homogenization of reactive transport in porous 
media \cite{AuAd:95, CM:08, CDT:03, vDK92, HJ1, Mau:91, MikVanD:05, MP:04, MP:06} and references therein. 
The present work is a sequel to \cite{Allaire10, AllaireHutridurga, AllaireMikelic10}: more precisely, 
it generalizes the homogenization of these previous linear models in a regime 
of strong convection to the nonlinear case of a so-called Langmuir isotherm for 
the reaction term. Of course, there are previous works on the homogenization 
of nonlinear models of reactive flows in porous media (see \cite{CDT:03, HJ1, HJM:94, MP:04, MP:06} 
to cite a few of them). However, to our knowledge, none of them were concerned 
with the present setting where, at the pore scale, convection, diffusion and 
reaction are of the same order of magnitude. Such a local equilibrium of all 
terms in the microscopic model yields a large convection at the macroscopic 
scale.

To be more specific, we now describe the main physical assumptions and give our 
detailed mathematical model. We consider a single solute dissolved in an incompressible 
saturated fluid in a porous medium. An adsorption/desorption reaction 
can occur at the pore boundaries. We use the Langmuir isotherm to model the reaction 
phenomenon. There are two scalar unknown concentrations of the solute: $u_\ve$ in the bulk 
and $v_\ve$ on the liquid/solid interfaces. A convection-diffusion equation is 
considered for $u_\ve$ in the bulk and a similar equation is considered for $v_\ve$ 
on the pore boundaries. These two equations are coupled using a term that represents 
the reaction process at the interfaces. Of course, in most of the applications, 
the assumption of single solute being dissolved in the fluid is far from reality. 
So, our model is a toy model and should by no means be considered complete. 
In a recent preprint \cite{AllaireHutridurga2} we have considered a more involved 
multiple species model.

We consider an $\ve$-periodic infinite porous medium where $\ve$ is a small 
positive parameter, defined as the ratio between the period and a characteristic 
macroscopic lengthscale. Typically, this medium is built out 
of $\mathbb{R}^d$ ($d=2$ or $3$, being the space dimension) 
by removing a periodic distribution of solid obstacles which, after rescaling, 
are all similar to the unit obstacle $\Sigma^0$. More precisely, let $Y = [0,1]^d$ 
be the unit periodicity cell. Let us consider a smooth partition: $Y = \Sigma^0 \cup Y^0$, 
where $\Sigma^0$ is the solid part and $Y^0$ is the fluid part. The unit periodicity 
cell is identified with the flat unit torus $\mathbb{T}^d$. The fluid part is assumed 
to be a smooth connected open subset whereas no particular assumptions are made on the solid part.

For each multi-index $j\in\mathbb{Z}^d$, we define $Y^j_\ve = \ve(Y^0+j)$, 
$\Sigma^j_\ve = \ve(\Sigma^0+j)$, $S^j_\ve = \ve(\partial\Sigma^0+j)$, the 
periodic porous medium $\Omega_\ve = \displaystyle \cup_{j\in\mathbb{Z}^d} Y^j_\ve$ 
and the $(d-1)$-dimensional surface $\partial\Omega_\ve = \cup_{j\in\mathbb{Z}^d} S^j_\ve$. 
The following standard notations in the theory of Homogenization are used: $x$ denotes 
the macroscopic space variable (running in $\Omega_\ve$ 
or in $\mathbb{R}^d$) and $y$ denotes the microscopic space variable (running in $Y$). 
We will often use the change of variables: $y=x/\ve$. 

We denote by $n(y)$ the exterior unit normal to $Y^0$ and by 
$d\sigma(y)$ the Lebesgue surface measure on $\partial Y^0=\partial\Sigma^0$. 
Then, $G(y) = Id - n(y) \otimes n(y)$ 
is the projection matrix on the tangent hyperplane to the surface 
$\partial Y^0=\partial\Sigma^0$. In order to define a Laplace-Beltrami operator on 
this surface, we define the tangential gradient $\nabla^s_y = G(y) \nabla_y$ and the 
tangential divergence $\div^s_y \Psi = \div_y (G(y) \Psi)$ for any $\Psi(y):\R^d\to\R^d$. 
Scaling the projection matrix using $y=x/\ve$ gives a projection matrix, $G_\ve(x)=G(x/\ve)$, 
on the tangent hyperplane to the pore boundary $\partial\Omega_\ve$ and, consequently, 
rescaled tangential operators, denoted by $\nabla^s$ and $\div^s$, are defined with respect 
to the $x$ variable on $\partial\Omega_\ve$.

We assume that the porous medium is saturated with an incompressible fluid, the velocity 
of which is assumed to be given, independent of time and periodic in space. The fluid 
cannot penetrate the solid obstacles but can slip on their surface. Therefore, we consider 
two periodic vector fields: $b(y)$, defined in the bulk $Y^0$, and $b^s(y)$, defined on 
the surface $\partial\Sigma^0$ and belonging at each point of $\partial\Sigma^0$ to its 
tangent hyperplane. Assuming that the fluid is incompressible and does not penetrate 
the obstacles means that
$$
\div_y b(y) = 0 \quad \mbox{ in } Y^0, \quad 
b(y) \cdot n(y) = 0 \quad \textrm{ on } \partial\Sigma^0 ,
$$
$$
\div^s_y b^s(y) = 0 \quad \textrm{ on } \partial\Sigma^0, \quad 
b^s(y) \cdot n(y) = 0 \quad \textrm{ on } \partial\Sigma^0.
$$
In truth, $b^s(y)$ should be the trace of $b(y)$ on $\partial\Sigma^0$ but, since this 
property is not necessary for our analysis, we shall not make such an assumption. 
Of course, some regularity is required for these vector fields and we assume that 
$b(y) \in L^\infty(Y^0;\mathbb{R}^d)$, $b^s(y) \in L^\infty(\partial \Sigma^0;\mathbb{R}^d)$.

We assume that the molecular diffusion is periodic, possibly anisotropic, varying in space 
and different in the bulk and on the pore boundaries. In other words, we introduce two periodic 
symmetric tensors $D(y)$ and $D^s(y)$, with entries belonging respectively to $L^\infty(Y^0)$ 
and to $L^\infty(\partial\Sigma^0)$, which are assumed to be uniformly coercive, namely that 
there exists a constant $C>0$ such that, for any $\xi\in\mathbb{R}^d$,

$$
D(y)\xi\cdot\xi \geq C|\xi|^2 \mbox{ a.e. in } Y^0 , \quad 
D^s(y)\xi\cdot\xi \geq C|\xi|^2 \mbox{ a.e. on } \partial\Sigma^0 .
$$

Let us introduce three positive constants, $\kappa$ (the adsorption rate), 
$\alpha$ and $\beta$ (the Langmuir parameters). For some positive final time $T$, 
let us consider the following coupled system of parabolic equations of which the scalar concentrations $u_\ve$ 
and $v_\ve$ are the solutions:
\begin{equation}
\label{eq:p-1}
\frac{\partial {u_\ve}}{\partial t} + \frac{1}{\ve}b_\ve \cdot \nabla u_\ve - \div\left(D_\ve \nabla {u_\ve}\right) = 0 \: \: \textrm{in} \: \: (0,T)\times\Omega_\ve,
\end{equation}
\begin{equation}
\label{eq:p-2}
\frac{\partial v_\ve}{\partial t} + \frac{1}{\ve}b^s_\ve \cdot \nabla^s v_\ve - \div^s\left(D^s_\ve \nabla^s {v_\ve}\right) = \frac{\kappa}{\ve^2} \left[\frac{\alpha u_\ve}{1+\beta u_\ve} - v_\ve \right] \: \: \textrm{on} \: \: (0,T)\times \partial \Omega_\ve,
\end{equation}
\begin{equation}
\label{eq:p-2b}
-\frac{D_\ve}{\ve} \nabla u_\ve \cdot n = \frac{\kappa}{\ve^2} \left[\frac{\alpha u_\ve}{1+\beta u_\ve} - v_\ve \right] \: \: \textrm{on} \: \: (0,T)\times \partial \Omega_\ve,
\end{equation}
\begin{equation}
\label{eq:p-3}
u_\ve(0,x) = u^{in}(x) \mbox{ in } \Omega_\ve , \quad v_\ve(0,x) = v^{in}(x) \mbox{ on } \partial\Omega_\ve,
\end{equation}
with the notations (and similar ones for $b^s$ and $D^s$):
$$
b_\ve(x) = b\left(\frac{x}{\ve}\right) \quad \mbox{ and } \quad 
D_\ve(x) = D\left(\frac{x}{\ve}\right) .
$$
The specific $\ve$-scaling of the coefficients in (\ref{eq:p-1})-(\ref{eq:p-2b}) 
is not new and is well explained, e.g., in \cite{AllaireMikelic10}. Before 
adimensionalization, the physical system of equations is written without any 
power of $\ve$ in the original time-space coordinates $(\tau,y)$. Since we are 
interested in a macroscopic view and a long time behaviour of this coupled 
system of equations, we perform a ``parabolic'' scaling of the time-space 
variables, namely $(\tau,y)\to(\ve^{-2}t,\ve^{-1}x)$, which precisely yields 
the scaled model (\ref{eq:p-1})-(\ref{eq:p-2b}).

The nonlinear Langmuir isotherm is denoted by $f$ and $F$ is its primitive such that 
$F'(u)=f(u)$ and $F(0)=0$, namely
\begin{equation}
\label{eq:fprim}
f(u_\ve) = \frac{\alpha u_\ve}{1+\beta u_\ve}, \:\:\:\: 
F(u_\ve) = \frac{\alpha}{\beta}\left[u_\ve - \frac{1}{\beta} \log(1+\beta u_\ve)\right] .
\end{equation}
The initial data are chosen non-negative: $u^{in}, v^{in} \geq 0$ and such that 
$u^{in} \in L^2(\mathbb{R}^d)\cap L^\infty(\mathbb{R}^d)$ 
and $v^{in}\in H^1(\mathbb{R}^d)\cap L^\infty(\mathbb{R}^d)$ 
so that its trace is well defined on $\partial\Omega_\ve$. In order to homogenize the system 
(\ref{eq:p-1})-(\ref{eq:p-3}), we need a technical assumption on the velocity fields which 
amounts to saying that the bulk and surface drifts are equal (their common value being called 
$b^*$ in the sequel)
\begin{equation}
\label{eq:drift}
\frac{1}{|Y^0|}\int_{Y^0}b(y)\, {\rm d}y = \frac{1}{|\partial\Sigma^0|}\int_{\partial\Sigma^0}b^s(y)\, {\rm d}\sigma(y) = b^* .
\end{equation}
Such an assumption was not necessary in the linear case \cite{AllaireHutridurga} but 
is the price to pay for extending our previous results to the nonlinear case of the 
Langmuir isotherm. 

Our main result (Theorem \ref{main-weak}) says that the solution $(u_\ve, v_\ve)$ of 
(\ref{eq:p-1})-(\ref{eq:p-3}) is approximately given by the ansatz:
$$
u_\ve(t,x) \approx u_0\left(t,x-\frac{b^*t}{\ve}\right)+\ve u_1\left(t,x-\frac{b^*t}{\ve},\frac{x}{\ve}\right) ,
$$
$$
v_\ve(t,x) \approx f(u_0)\left(t,x-\frac{b^*t}{\ve}\right)+\ve v_1\left(t,x-\frac{b^*t}{\ve},\frac{x}{\ve}\right) ,
$$
where $u_0$ is the solution of the following macroscopic nonlinear diffusion equation:
$$
\left\{
\begin{array}{l}
\dsp \left[|Y^0| + \frac{\alpha|\partial\Sigma^0|}{(1+\beta u_0)^2}\right] 
\frac{\partial u_0}{\partial t} - \div_x(A^*(u_0)\nx u_0) = 0
\quad \textrm{ in } (0,T)\times\mathbb{R}^d ,\\[0.5cm] 
\dsp \left[|Y^0|u_0 + \frac{|\partial\Sigma^0|\: \alpha\: u_0}{1+\beta u_0}\right](0,x) = |Y^0|u^{in}(x) + |\partial\Sigma^0| v^{in}(x) \quad \textrm{ in } \mathbb{R}^d ,
\end{array} \right.
$$
and the corrector $(u_1,v_1)$ are defined by
$$
u_1(t,x,y) = \chi\big(y, u_0(t,x)\big)\cdot \nx u_0(t,x)
$$
and
$$
v_1(t,x,y) = \frac{\alpha}{(1+\beta u_0(t,x))^2} \omega\big(y, u_0(t,x)\big)\cdot \nx u_0(t,x)
$$
where $(\chi,\omega) = (\chi_i,\omega_i)_{1\leq i \leq d}$ is the solution of the cell 
problem:
$$
\left\{ 
\begin{array}{ll}
-b^* \cdot e_i + b(y)\cdot(e_i + \ny \chi_i) - \div_y(D(e_i + \ny \chi_i)) = 0 & \textrm{in } Y^0,\\[0.3cm]
- D\left( e_i + \ny \chi_i\right)\cdot n = \dsp\frac{\alpha \kappa}{(1+\beta u_0)^2} \left(\chi_i - \omega_i \right) & \textrm{on } \partial \Sigma^0, \\[0.3cm]
-b^* \cdot e_i + b^s(y)\cdot(e_i + \ny^s \omega_i) - \div^s_y(D^s(e_i+ \nabla^s_y \omega_i)) = \kappa \left(\chi_i -  \omega_i \right) & \textrm{on } \partial \Sigma^0, \\[0.3cm]
y \to (\chi_i(y),\omega_i(y)) \quad  Y-\textrm{periodic.} &
 \end{array}\right.
$$
Note that the cell solution $(\chi,\omega)$ depends not only on $y$ but 
also on the value of $u_0(t,x)$. Furthermore, the technical assumption (\ref{eq:drift}) 
is precisely the compatibility condition for solving the cell problem for 
any value of $u_0(t,x)$. The obtained ansatz indicates that, in the limit, the 
bulk and surface concentrations are in equilibrium since the leading term 
for $v_\ve$ is $f(u_0)$ where $u_0$ is the leading term for $u_\ve$. 
Eventually, the effective diffusion 
(or dispersion) tensor $A^*(u_0)$ is given by
$$
\begin{array}{ll}
\dsp A_{ij}^*(u_0) = &\dsp \iy D(y) \left( \ny\chi_i + e_i\right) \cdot \left(\ny\chi_j + e_j\right) \, {\rm d}y \\[0.3cm]
& + \dsp\frac{\alpha\kappa}{(1+\beta u_0)^2}\ip \left[\chi_i - \omega_i\right]\left[\chi_j - \omega_j\right] \, {\rm d}\sigma(y)\\[0.3cm]
& + \dsp\frac{\alpha}{(1+\beta u_0)^2} \ip D^s(y) \left( \ny^s \omega_i+e_i\right) \cdot 
\left(\ny^s\omega_j + e_j\right) \, {\rm d}\sigma(y)\\[0.3 cm]
& + \dsp \iy D(y)\Big(\ny\chi_j\cdot e_i - \ny\chi_i\cdot e_j\Big)\, {\rm d}y\\[0.3 cm]
& +\dsp\frac{\alpha}{(1+\beta u_0)^2} \ip D^s(y)\Big(\ny^s\omega_j\cdot e_i - \ny^s\omega_i\cdot e_j\Big)\, {\rm d}\sigma(y)\\[0.3 cm]
&\dsp+ \iy \Big(b(y)\cdot\ny\chi_i\Big)\chi_j\, {\rm d}y +\frac{\alpha}{(1+\beta u_0)^2} \ip \Big(b^s(y)\cdot\ny^s\omega_i\Big)\omega_j\, {\rm d}\sigma(y).
\end{array}
$$
Remark that the dispersion matrix $A^*$ is neither symmetric nor a constant matrix. 
The fact that the non-linearity passed from the reaction term at the microscopic 
level to the diffusion term at the macroscopic one is another manifestation of 
the strong coupling of convection, diffusion and reaction in the homogenization 
process. For small values of the concentration $u_0$, the dispersion tensor $A^*(u_0)$ 
is close to the one obtained in the linear case. However, for large values of $u_0$, 
the saturation effect of the Langmuir isotherm implies that the entries of $A^*(u_0)$ 
are much smaller with a finite positive asymptote (see (\ref{formal-limit-u0-infin}) 
and the discussion in Section \ref{sec:num}). 
 
This article is outlined as follows. Section \ref{sec:max} deals with the maximum principle 
(see Proposition \ref{prop:max}) and uniform a priori estimates on the solutions of (\ref{eq:p-1})-(\ref{eq:p-3}) 
which are obtained via energy estimates (see Lemma \ref{lem:apriori}). 
In passing, the obtained a priori estimates yield existence 
and uniqueness of the solution of (\ref{eq:p-1})-(\ref{eq:p-3}) by standard 
arguments relying on the monotone character of the Langmuir isotherm 
(see Proposition \ref{prop:exist}). 
The non-linearity of (\ref{eq:p-1})-(\ref{eq:p-3}) requires some strong compactness 
of the sequence of solutions in order to pass to the limit. This is obtained in 
Corollary \ref{cor:comp-u} which is the most technical result of the present 
paper. Following the ideas of \cite{Marusic05, AllairePiatnitski}, we 
first show that, in a moving frame of reference, a uniform localization of 
solution holds (Lemma \ref{lem:localization}). Then a time equicontinuity 
type result (Lemma \ref{lem:combi}) allows us to gain compactness. These 
technical results are not straightforward extensions of those in \cite{Marusic05, 
AllairePiatnitski}. There are a number of additional difficulties, 
including the perforated character of the domain, the non-linearity of the 
equations and more importantly the fact that there are two unknowns $u_\ve$ 
and $v_\ve$. 
Section \ref{sec:2sc} is dedicated to the derivation of the homogenized equation 
(Theorem \ref{main-weak}) using the method of two-scale convergence with drift 
\cite{Marusic05, Allaire08}. The essence of this method is briefly recalled 
in Propositions \ref{compact} to \ref{h1-bdry-conv}. 
Theorem \ref{main-weak} gives a result of weak convergence of the sequence 
$(u_\ve, v_\ve)$ to the homogenized limit $(u_0, v_0=f(u_0))$. Although the 
previous Corollary \ref{cor:comp-u} gives some strong compactness in the 
$L^2$-norm, there is still room to improve the strong convergence, notably 
for the gradients of $u_\ve$ and $v_\ve$. This is the purpose of 
Section \ref{sec:strong} where we establish a strong convergence result 
(Theorem \ref{thm:strong}) for well prepared initial data. 
Eventually, Section \ref{sec:num} is devoted to some numerical simulations in two space dimensions using 
the FreeFem++ package \cite{Pirofreefem}. In the $2D$ setting, assumption (\ref{eq:drift}) 
implies that the homogenized drift vanishes i.e., $b^*=0$. We study the behavior of the 
homogenized dispersion tensor with respect to variations of the magnitude of $u_0$, 
the reaction rate $\kappa$ and the surface molecular diffusion $D^s$.
The results of the present paper are part of the PhD thesis of the second author 
which contains additional details, see \cite{Harsha}.

\section{Maximum principles and a priori estimates}
\label{sec:max}

The goal of this section is to prove a maximum principle, to derive uniform (with respect to $\ve$) a priori estimates based 
on energy equality and to deduce an existence and uniqueness result for the 
solution of (\ref{eq:p-1})-(\ref{eq:p-3}).

\begin{defn}
\label{defn:wsoln}
A pair $(u_\ve, v_\ve)\in L^2((0,T);H^1(\Omega_\ve))\times L^2((0,T);H^1(\partial\Omega_\ve))$ with
 $\Big(\dsp\frac{\partial u_\ve}{\partial t}, \frac{\partial v_\ve}{\partial t}\Big)\in L^2((0,T);(H^1(\Omega_\ve))')\times L^2((0,T);(H^1(\partial\Omega_\ve))')$ is said to be a weak solution of the coupled system (\ref{eq:p-1})-(\ref{eq:p-3}) provided we have
\begin{equation}
\label{eq:wsoln:in}
(u_\ve, v_\ve)(0)=(u^{in}, v^{in})
\end{equation}
and for a.e. time $0\le t\le T$, we have
\begin{equation}
\label{eq:wsoln}
\begin{array}{cc}
\dsp \int_{\Omega_\ve} \frac{\partial u_\ve}{\partial t} \phi \, {\rm d}x + \frac{1}{\ve}\int_{\Omega_\ve} b_\ve \cdot \nabla u_\ve \phi\, {\rm d}x
+ \int_{\Omega_\ve} D_\ve \nabla u_\ve \cdot \nabla \phi\, {\rm d}x\\[0.3 cm]
\dsp+ \ve \int_{\partial\Omega_\ve} \frac{\partial v_\ve}{\partial t}\psi \, {\rm d}\sigma(x) + \int_{\partial\Omega_\ve} b^s_\ve \cdot \nabla^s v_\ve \psi\, {\rm d}\sigma(x) + \ve\int_{\partial\Omega_\ve} D_\ve^s \nabla^s v_\ve \cdot \nabla^s \psi\, {\rm d}\sigma(x)\\[0.2 cm]
\dsp+ \frac{\kappa}{\ve}\int_{\partial\Omega_\ve}\left(f(u_\ve) - v_\ve\right) \left(\phi - \psi\right)\, {\rm d}\sigma(x)= 0,
\end{array}
\end{equation}
for each pair $(\phi, \psi)\in H^1(\Omega_\ve)\times H^1(\partial\Omega_\ve)$.
\end{defn}

In (\ref{eq:wsoln}) ${\rm d}\sigma(x)$ is the surface measure on $\partial\Omega_\ve$. 
In truth the integrals of the time derivatives in (\ref{eq:wsoln}) should be replaced by 
the corresponding duality pairings of $(H^1(\Omega_\ve))'$ and $H^1(\Omega_\ve)$ on the one hand, 
and of $(H^1(\partial\Omega_\ve))'$ and $H^1(\partial\Omega_\ve)$ on the other hand. 
We indulge ourselves with this usual abuse of notations which simplify the exposition. 
By the well-known Aubin-Lions lemma, the solution is continuous in time, namely 
$u_\ve\in C([0,T];L^2(\Omega_\ve))$ and $v_\ve\in C([0,T];L^2(\partial\Omega_\ve))$, 
so that the initial condition makes sense in (\ref{eq:wsoln:in}).

\begin{prop}
\label{prop:exist}
Assume that the initial data $(u^{in},v^{in})$ belong to the space 
$L^2(\mathbb{R}^d)\cap L^\infty(\mathbb{R}^d) \times H^1(\mathbb{R}^d)\cap L^\infty(\mathbb{R}^d)$ 
and are non-negative. 
There exists a unique weak solution $(u_\ve,v_\ve)$ of (\ref{eq:p-1})-(\ref{eq:p-3}) 
in the sense of Definition \ref{defn:wsoln}.
\end{prop}

The above existence result relies on a maximum principle that we shall 
prove assuming that a weak solution of (\ref{eq:p-1})-(\ref{eq:p-3}) exists. 
We use the standard notations: $h^+ = \max(0,h)$ and $h^- = \min(0,h)$. 
Recall that the function $f(u)=\alpha u/(1+\beta u)$ is one to one and 
increasing from $\R^+$ to $[0,\alpha/\beta]$. Although the function $f(u)$ 
is not defined for $u=-1/\beta$, and since we are interested only in 
non-negative values of $u$, we can modify and mollify $f(u)$ for $u<0$ 
so that it is an increasing function on $\R$ which grows at most linearly 
at infinity with a uniformy bounded derivative. 
With this modification, all computations below make sense for negative values 
of $u$. In particular, if $u$ is a function in $H^1(\mathbb{R}^d)$ so is $f(u)$. 

\begin{prop}
\label{prop:max}
Let $(u_\ve,v_\ve)$ be a weak solution of (\ref{eq:p-1})-(\ref{eq:p-3}) in the 
sense of Definition \ref{defn:wsoln}. Assume that the initial data 
$(u^{in},v^{in})$ satisfy $0\leq u^{in} \leq M_u$, $0\leq v^{in}\leq M_v$ 
for some positive constants $M_u$ and $M_v$ (without loss of generality consider $f(M_u)<M_v$).

If $M_v<\alpha/\beta$, then 
$$
\left\{ 
\begin{array}{ll}
0\leq u_\ve(t,x) \leq m_u=f^{-1}(M_v) & \mbox{ for } (t,x)\in(0,T)\times\Omega_\ve ,\\
0\leq v_\ve(t,x)\leq M_v & \mbox{ for } (t,x)\in(0,T)\times\partial\Omega_\ve .
\end{array}\right.
$$

If $M_v\geq\alpha/\beta$, then there exist three positive constants $\tau$, $M(\tau)$ 
and $\tilde M_v<\alpha/\beta$, independent of $\ve$, such that
$$
\left\{ 
\begin{array}{ll}
0\leq u_\ve(t,x) \leq M(\tau) & \mbox{ for } (t,x)\in(0,\ve^2\tau)\times\Omega_\ve , \\
0\leq v_\ve(t,x)\leq M_v & \mbox{ for } (t,x)\in(0,\ve^2\tau)\times\partial\Omega_\ve ,
\end{array}\right.
$$
and
$$
\left\{ 
\begin{array}{ll}
0\leq u_\ve(t,x) \leq \tilde m_u=f^{-1}(\tilde M_v) & \mbox{ for } (t,x)\in(\ve^2\tau,T)\times\Omega_\ve ,\\
0\leq v_\ve(t,x)\leq\tilde  M_v & \mbox{ for } (t,x)\in(\ve^2\tau,T)\times\partial\Omega_\ve .
\end{array}\right.
$$
\end{prop}

\begin{proof}
We use a variational approach. To begin with, we prove that the solutions remain 
non-negative for non-negative initial data. Let us consider $({f(u_\ve)}^-, \ve v_\ve^-)$ 
as test functions in the variational formulation of (\ref{eq:p-1})-(\ref{eq:p-3}):
$$
\iti\frac{{\rm d}}{{\rm d}t} \id  {F(u^-_\ve)} \, {\rm d}x \, {\rm d}t + 
\frac{\ve}{2}\iti\frac{{\rm d}}{{\rm d}t}\ib |v^-_\ve|^2 \, {\rm d}\sigma(x)\, {\rm d}t +
\frac{1}{\ve}\iti\id b_\ve \cdot \nabla {F(u^-_\ve)} \, {\rm d}x\, {\rm d}t
$$
$$
- \iti\id \div\left(D_\ve \nabla u_\ve\right) {f(u_\ve)}^- \, {\rm d}x\, {\rm d}t + 
\frac{\ve}{2}\iti\ib b_\ve^s \cdot \nabla^s |v_\ve^-|^2 \, {\rm d}\sigma(x)\, {\rm d}t 
$$
$$
-\ve\iti\ib \div^s\left(D_\ve^s \nabla^s v_\ve\right) v_\ve^- \, {\rm d}\sigma(x)\, {\rm d}t
- \frac{\kappa}{\ve} \iti\ib \left[f(u_\ve) - v_\ve \right] v_\ve^- \, {\rm d}\sigma(x)\, {\rm d}t=0 ,
$$
where $f$ and its primitive $F$ are defined by (\ref{eq:fprim}). 
The convective terms in the above expression vanish due to the divergence free property 
of $b, b^s$ and the boundary condition $b\cdot n=0$ on $\partial\Sigma_0$. Thus, we get
$$
\id  {F(u^-_\ve)}(T) \, {\rm d}x +\frac{\ve}{2}\ib |v^-_\ve(T)|^2 \, {\rm d}\sigma(x) + \iti\id f'(u^-_\ve) D_\ve \nabla u^-_\ve\cdot \nabla u^-_\ve \, {\rm d}x\, {\rm d}t 
$$
$$
+ \ve\iti\ib D_\ve^s \nabla^s v^-_\ve \cdot \nabla^s v_\ve^- \, {\rm d}\sigma(x)\, {\rm d}t + \frac{\kappa}{\ve} \iti\ib \left[f(u_\ve) - v_\ve \right] \left[{f(u_\ve)}^- - v_\ve^- \right] \, {\rm d}\sigma(x)\, {\rm d}t
$$
$$
= \id F(u^-_\ve)(0) \, {\rm d}x + \frac{\ve}{2} \ib |v_\ve^-(0)|^2 \, {\rm d}\sigma(x) .
$$
Since the function $h \to h^-$ is monotone and $f'(u)\geq0$, all terms on the 
left hand side of the above equation are non-negative. The assumption of the non-negative initial data implies 
that the right hand side vanishes, therefore proving that $u^-_\ve(t,x) = 0$, $v^-_\ve(t,x) = 0$. 
Thus the solutions $u_\ve$ and $v_\ve$ stay non negative at all times. 

Next, we show that the solutions stay bounded from above if we start with a bounded initial data. 
The boundedness property of $f$ adds an additional difficulty prompting us to consider two cases as below.

{\bf Case I.} Assume $M_v<\alpha/\beta$ so that we can define $m_u=f^{-1}(M_v)>M_u$.

We choose $(({f(u_\ve)} - M_v)^+, \ve (v_\ve - M_v)^+)$ as test functions in the variational 
formulation of (\ref{eq:p-1})-(\ref{eq:p-3}). Introducing the primitive function ${\mathcal F}$ 
such that ${\mathcal F}'(u) = (f(u)-M_v)^+$ and ${\mathcal F}(0) =0$, we get
$$
\id {\mathcal F}(u_\ve)(T)\, {\rm d}x + \frac{\ve}{2} \ib |(v_\ve - M_v)^+|^2 \, {\rm d}\sigma(x) 
$$
$$
+ \iti\id f'(u_\ve) D_\ve \nabla (u_\ve - m_u)^+ \cdot \nabla (u_\ve - m_u)^+ \, {\rm d}x \, {\rm d}t
$$
$$
 + \ve\iti\ib D_\ve^s\nabla^s (v_\ve - M_v)^+ \cdot \nabla^s (v_\ve - M_v)^+ \, {\rm d}\sigma(x)\, {\rm d}t 
$$
$$
+ \frac{\kappa}{\ve} \iti\ib \left[(f(u_\ve) - M_v) - (v_\ve - M_v) \right] \left[({f(u_\ve)} -  M_v)^+ - (v_\ve - M_v)^+ \right] \, {\rm d}\sigma(x)\, {\rm d}t  
$$
$$
= \id {\mathcal F}(u^{in})\, {\rm d}x + \frac{\ve}{2} \ib ((v^{in} - M_v)^+)^2 \, {\rm d}\sigma(x) ,
$$
because $\nabla ({f(u_\ve)} - M_v)^+ = f'(u_\ve) \nabla (u_\ve - m_u)^+$. 
The upper bound on the initial data implies 
that the right hand side vanishes. The left hand side is non-negative 
because $h \to h^+$ is monotone and $f'(u_\ve)\geq0$. Since ${\mathcal F}(u) = 0$ 
if and only if $u \leq m_u$, we deduce that $u_\ve\leq m_u$ and $v_\ve \leq M_v$.

{\bf Case II.} Assume $M_v\geq\alpha/\beta$.

The argument in Case I fails because $f^{-1}(M_v)$ is not well defined. 
The idea is to first prove that there exists $\tau>0$ such that, after a 
short time $\ve^2\tau$, the solution $v_\ve$ reduces in magnitude and is 
uniformly smaller than $\alpha/\beta$. Then by taking $\ve^2\tau$ as a 
new initial time we can repeat the analysis of Case I. 
Whatever the value of $u_\ve$, equation (\ref{eq:p-2}) implies the following 
inequality on the boundary $\partial\Omega_\ve$:
\begin{equation}
\label{eq:firststep}
\frac{\partial v_\ve}{\partial t} + \frac{1}{\ve}b^s_\ve \cdot \nabla^s v_\ve - \div^s\left(D^s_\ve \nabla^s {v_\ve}\right) \leq \frac \kappa {\ve^2}\Big(\frac \alpha \beta - v_\ve\Big) .
\end{equation}
Choosing $w_\ve = \left( v_\ve - \frac \alpha \beta \right) \exp(t\kappa/\ve^2)$ we get
$$
\frac{\partial w_\ve}{\partial t} + \frac{1}{\ve}b^s_\ve \cdot \nabla^s w_\ve - \div^s\left(D^s_\ve \nabla^s {w_\ve}\right) \leq 0
$$
with the initial data
$$
w_\ve(0)= v^{in} - \frac \alpha \beta .
$$
Then, the maximum principle implies (see \cite{Protter84} for details)
$$
w_\ve(t)\leq \max_{x\in\partial\Omega_\ve} w_\ve(0) \leq M_v - \frac \alpha \beta ,
$$
which yields the following upper bound:
\begin{equation}
\label{eq:vbound}
v_\ve(t) \leq \exp(-t\kappa/\ve^2) M_v 
+ \Big(1 - \exp(-t\kappa/\ve^2) \Big) \frac \alpha \beta . 
\end{equation}
Unfortunately (\ref{eq:vbound}) is too crude a bound which cannot reduce the 
initial bound $M_v$ to a number smaller than $\alpha/ \beta$. At least, (\ref{eq:vbound}) 
yields $v_\ve(t) \leq M_v$. We are going to use this upper bound in the equation for 
$u_\ve$ in order to improve (\ref{eq:vbound}). 

Equations (\ref{eq:p-1}) and (\ref{eq:p-2b}) are:
$$
\frac{\partial {u_\ve}}{\partial t} + \frac{1}{\ve}b_\ve \cdot \nabla u_\ve - \div\left(D_\ve \nabla {u_\ve}\right) = 0 \quad \mbox{ in } \Omega_\ve ,
$$ 
$$
-\frac{D_\ve}{\ve} \nabla u_\ve \cdot n = \frac{\kappa}{\ve^2} \left[\frac{\alpha u_\ve}{1+\beta u_\ve} - v_\ve \right] =: \frac{\kappa}{\ve^2} g_\ve \quad \mbox{ on } \partial\Omega_\ve ,
$$
with $g_\ve$ satisfying, thanks to (\ref{eq:vbound}), the bound $|g_\ve|\leq M_v$.
Let us introduce an auxiliary problem in the unit cell:
\begin{equation}
\left\{
\begin{array}{ll}
b\cdot\nabla_y \Psi - \div_y(D\nabla_y \Psi) = \kappa M_v |\partial\Sigma^0|/|Y^0| &\textrm{in} \: \: Y^0, \\
-D(y)\nabla_y \Psi\cdot n = \kappa M_v &\textrm{on} \: \: \partial\Sigma^0, \\
y \to \Psi(y) & \textrm{is} \: \:Y-\textrm{periodic,}
\end{array} \right.
\label{eq:auxmax}
\end{equation}
where the compatibility condition for the existence and uniqueness (up to an additive constant) 
of $\Psi$ is satisfied. The scaled function $\Psi_\ve(x)=\Psi(x/\ve)$ satisfies:
\begin{equation}
\left\{
\begin{array}{ll}
\dsp \frac 1\ve b_\ve\cdot\nabla \Psi_\ve - \div(D_\ve\nabla \Psi_\ve) = \frac{\kappa M_v}{\ve^2} |\partial\Sigma^0|/|Y^0| &\textrm{in} \: \: \Omega_\ve, \\[0.4cm]
\dsp -\frac 1 \ve D_\ve\nabla \Psi_\ve\cdot n = \frac{\kappa M_v}{\ve^2} &\textrm{on} \: \: \partial\Omega_\ve .
\end{array} \right.
\label{eq:auxmaxscl}
\end{equation}
The function $z_\ve = \Psi_\ve + u_\ve - (\frac{\kappa M_v}{\ve^2}|\partial\Sigma^0|/|Y^0|) t$ 
satisfies:
\begin{equation}
\left\{
\begin{array}{ll}
\dsp \frac{\partial {z_\ve}}{\partial t} + \frac{1}{\ve}b_\ve \cdot \nabla z_\ve - \div\left(D_\ve \nabla {z_\ve}\right) = 0 &\textrm{in} \: \: (0,T)\times\Omega_\ve, \\[0.4cm]
\dsp - \frac{1}{\ve} D_\ve\nabla z_\ve \cdot n = \frac{\kappa}{\ve^2} (g_\ve + M_v) \geq 0 &\textrm{on} \: \: (0,T)\times\partial\Omega_\ve, \\[0.4cm]
z_\ve(0) = \Psi_\ve + u^{in} &\textrm{in} \: \: \Omega_\ve .
\end{array} \right.
\label{eq:utilde}
\end{equation}
Then, the maximum principle yields (again, see \cite{Protter84} if necessary)
$$
z_\ve(t) \leq \max_{\Omega_\ve} ( \Psi_\ve + u^{in} ) \leq M_u + \|\Psi\|_{L^\infty(Y^0)} .
$$
From the definition of $z_\ve$ we deduce:
\begin{equation}
\label{eq:uboundeps}
u_\ve(t) \leq  M_u + 2 \|\Psi\|_{L^\infty(Y^0)}
+ \frac{\kappa M_v}{\ve^2} |\partial\Sigma^0|/|Y^0| t ,
\end{equation}
which implies that, for any $\tau>0$, we have 
$$
\max_{0<t<\ve^2\tau} u_\ve(t) \leq M(\tau) =  M_u + 2 \|\Psi\|_{L^\infty(Y^0)}
+ {\kappa M_v} |\partial\Sigma^0|/|Y^0| \tau ,
$$
where $M(\tau)$ does not depend on $\ve$ and is an affine function of $\tau$. Hence
$$
\max_{0<t<\ve^2\tau} f(u_\ve(t)) \leq f(M(\tau)) = \frac{\alpha M(\tau)}{1+\beta M(\tau)} 
< \frac\alpha\beta
$$
and (\ref{eq:firststep}) can be improved as 
$$
\frac{\partial v_\ve}{\partial t} + \frac{1}{\ve}b^s_\ve \cdot \nabla^s v_\ve - \div^s\left(D^S_\ve \nabla^s {v_\ve}\right) \leq \frac \kappa {\ve^2}\Big( f(M(\tau)) - v_\ve\Big) 
\quad \mbox{ for } 0<t<\ve^2\tau .
$$
The same argument leading to (\ref{eq:vbound}) now gives that, for any $\tau>0$ 
and $0<t<\ve^2\tau$, 
\begin{equation}
\label{eq:vboundeps}
v_\ve(t) \leq \exp(-\kappa\tau) M_v + \Big( 1 - \exp(-\kappa\tau) \Big) f(M(\tau)) , 
\end{equation}
where $f(M(\tau)) <  \frac\alpha\beta - \frac C \tau$ for some positive constant $C>0$. 
Thus, choosing $\tau$ large enough, we deduce that there exists $\tilde M_v$ (equal to 
the right hand side of (\ref{eq:vboundeps})) which does 
not depend on $\ve$ such that
$$
\max_{0<t<\ve^2\tau} v_\ve(t) \leq \tilde M_v <  \frac\alpha\beta .
$$
Choosing $\tilde M_u = M(\tau)$, we obviously have $f(\tilde M_u) < \tilde M_v$ 
and we can repeat the argument of Case I with the new initial time $\ve^2\tau$.
\end{proof}

We know from Proposition \ref{prop:max} that the solutions of (\ref{eq:p-1})-(\ref{eq:p-3}) 
are uniformly bounded in the $L^\infty$-norm. This shall help us obtain uniform (with respect to $\ve$) a priori energy estimates. 

\begin{lem}
\label{lem:apriori}
Let $(u_\ve,v_\ve)$ be a weak solution of (\ref{eq:p-1})-(\ref{eq:p-3}) in the sense of Definition \ref{defn:wsoln} and the initial data 
$(u^{in},v^{in})$ be such that $0\leq u^{in} \leq M_u$, $0\leq v^{in}\leq M_v$. 
There exists a constant $C$ that depends on $M_u$ and $M_v$ but not on $\ve$ such that
\begin{equation}
\label{eq:apriori}
\begin{array}{ll}
\dsp \|u_\ve\|_{L^\infty((0,T);L^2(\Omega_\ve))} + \sqrt{\ve} \|v_\ve\|_{L^\infty((0,T);L^2(\partial\Omega_\ve))} \\[0.4cm]
\dsp + \|\nabla u_\ve\|_{L^2((0,T)\times\Omega_\ve)} 
+ \sqrt{\ve}\|\nabla^s v_\ve\|_{L^2((0,T)\times\partial\Omega_\ve)} \\[0.4cm]
\dsp + \sqrt{\ve} \|w_\ve\|_{L^\infty((0,T);L^2(\partial\Omega_\ve))} \leq C 
\left( \left\|u^{in}\right\|_{L^2(\mathbb{R}^d)} + \|v^{in}\|_{H^1(\mathbb{R}^d)}\right) ,
\end{array}
\end{equation}
where $w_\ve = \ve^{-1} \left(\frac{\alpha u_\ve}{1+\beta u_\ve} - v_\ve\right)$. 
\end{lem}

\begin{proof}
To obtain an energy equality we multiply (\ref{eq:p-1}) by $f(u_\ve)$ and integrate 
over $\Omega_\ve$:
$$
\frac{{\rm d}}{{\rm d}t} \id F(u_\ve) \, {\rm d}x + \id f'(u_\ve) D_\ve \nabla u_\ve \cdot \nabla u_\ve \, {\rm d}x
+ \frac{\kappa}{\ve} \ib \left(f(u_\ve) - v_\ve\right) f(u_\ve) \, {\rm d}\sigma(x) = 0 ,
$$
where $F$ is the primitive of $f$, defined by (\ref{eq:fprim}), which satisfies 
$F(u)\geq0$ for $u\geq0$. 
We next multiply (\ref{eq:p-2}) by $\ve v_\ve$ and integrate over $\partial\Omega_\ve$:
$$
\frac{\ve}{2} \frac{{\rm d}}{{\rm d}t} \ib |v_\ve|^2 \, {\rm d}\sigma(x) 
+ \ve \ib D^s_\ve\nabla^s v_\ve\cdot\nabla^s v_\ve \, {\rm d}\sigma(x) 
- \frac{\kappa}{\ve} \ib \left(f(u_\ve) - v_\ve\right) v_\ve \, {\rm d}\sigma(x) = 0 .
$$
Adding the above two expressions leads to the following energy equality:
$$
\frac{{\rm d}}{{\rm d}t} \id F(u_\ve) \, {\rm d}x + \frac{\ve}{2} \frac{{\rm d}}{{\rm d}t} \ib |v_\ve|^2 \, {\rm d}\sigma(x) 
+ \id f'(u_\ve) D_\ve \nabla u_\ve \cdot \nabla u_\ve \, {\rm d}x 
$$
\begin{equation}
\label{eq:energy}
+ \ve \ib D^s_\ve\nabla^s v_\ve\cdot\nabla^s v_\ve \, {\rm d}\sigma(x) 
+ \frac{\kappa}{\ve} \ib \left(f(u_\ve) - v_\ve\right)^2 \, {\rm d}\sigma(x) = 0 .
\end{equation}
Recalling that, because of the maximum principle of Proposition \ref{prop:max}, 
$F(u_\ve)\geq0$ and $f'(u_\ve)\geq0$, and integrating over time yields:
$$
\left\|F(u_\ve)\right\|_{L^\infty((0,T);L^1(\Omega_\ve))} 
+ \ve \|v_\ve\|^2_{L^\infty((0,T);L^2(\partial\Omega_\ve))}
+ \ve \left\|\frac{1}{\ve}\left(f(u_\ve) - v_\ve\right)\right\|^2_{L^2(\partial\Omega_\ve \times (0,T))}
$$
$$
+ \left\|\sqrt{f'(u_\ve)} \nabla u_\ve\right\|^2_{L^2(\Omega_\ve \times (0,T))} 
+ \ve \|\nabla^s v_\ve\|^2_{L^2(\partial\Omega_\ve \times (0,T))}
$$
$$
\leq C \left(\left\|F(u^{in})\right\|_{L^1(\mathbb{R}^d)} + \|v^{in}\|^2_{H^1(\mathbb{R}^d)}\right) .
$$
A second-order Taylor expansion at 0 yields $F(u) = \frac12 u^2 f'(c)$ for some $c \in (0,u)$. 
By the maximum principle, we have $0\leq u_\ve \leq M$ so that
$$
0 < \frac{\alpha}{1+\beta M} \leq f'(u_\ve) \leq \alpha , 
$$
and $\left\|F(u^{in})\right\|_{L^1(\mathbb{R}^d)} \leq C \left\|u^{in}\right\|^2_{L^2(\mathbb{R}^d)}$ while 
$$
\left\|F(u_\ve)\right\|_{L^\infty((0,T);L^1(\Omega_\ve))} 
+ \left\|\sqrt{f'(u_\ve)} \nabla u_\ve\right\|^2_{L^2(\Omega_\ve \times (0,T))} 
\geq 
$$
$$
C \left( \left\|u_\ve\right\|^2_{L^\infty((0,T);L^2(\Omega_\ve))} 
+ \left\|\nabla u_\ve\right\|^2_{L^2(\Omega_\ve \times (0,T))} \right)
$$
from which we deduce (\ref{eq:apriori}). 
\end{proof}

\begin{proof}[Proof of Proposition \ref{prop:exist}]
From Proposition \ref{prop:max} and Lemma \ref{lem:apriori} it is a classical 
matter to prove existence and uniqueness of the weak solution of (\ref{eq:p-1})-(\ref{eq:p-3}). 
Existence can be proved, for example, by a finite dimensional Galerkin approximation 
\cite{Lady68, Lions69}, 
while uniqueness is a consequence of the monotonicity of the Langmuir isotherm or 
of its globally Lipschitz property. Since these arguments are well known (see, e.g., 
for a very similar model, \cite{MP:06}), we do not reproduce them here (see \cite{Harsha}  
for details, if necessary). 
\end{proof}

In order to find the homogenized limit for (\ref{eq:p-1})-(\ref{eq:p-3}), 
we need to pass to the limit in its variational formulation as $\ve\to0$, 
which requires some strong compactness since (\ref{eq:p-1})-(\ref{eq:p-3}) 
is nonlinear.  
The a priori estimates of Lemma \ref{lem:apriori} allow us to extract weakly 
converging subsequences but they do not give any strong compactness for the sequence 
$(u_\ve,v_\ve)$ since uniform (with respect to $\ve$) a priori estimates on their time derivatives are lacking. 
Furthermore, since $\Omega_\ve$ is unbounded, Rellich theorem does not hold in 
$\Omega_\ve$ and a localization result is thus required to get compactness. 
This is the goal of the results to follow for the rest of this section which 
culminate in Corollary \ref{cor:comp-u}. Their proofs rely on the use of 
the equations (\ref{eq:p-1})-(\ref{eq:p-3}). There is however one additional 
hurdle which is the presence of large convective terms of order $\ve^{-1}$. 
In order to compensate this large drift, following the lead of \cite{Marusic05}, 
we shall prove these compactness and localization results, not for the original 
sequence $(u_\ve,v_\ve)$, but for its counterpart defined in a moving frame of 
reference. For $\varphi_\ve(t,x)$, let us define its counterpart in moving coordinates as
\begin{equation}
\label{def.drift}
\drift{\varphi}_\ve(t,x) = \varphi_\ve\left(t,x+\frac{b^*t}{\ve}\right) , 
\end{equation}
where $b^*$ is the effective drift defined by (\ref{eq:drift}). 
Of course, definition (\ref{def.drift}) is consistent with the notion of two-scale 
convergence with drift which shall be recalled in Section \ref{sec:2sc} (in particular, 
if $\varphi_\ve(t,x)=\varphi\left(t,x-\frac{b^*t}{\ve}\right)$ is a test function 
for two-scale convergence with drift, then $\drift{\varphi}_\ve=\varphi$). 
A ``symmetric" definition will be useful too:
\begin{equation}
\label{def.mdrift}
\mdrift{\varphi}_\ve(t,x) = \varphi_\ve\left(t,x-\frac{b^*t}{\ve}\right) .
\end{equation}
As we consider functions in moving coordinates, the underlying porous domain 
$\Omega_\ve$ does ``move" with the same velocity. Let us define
\begin{equation}
\label{eq:def-mov-domain}
\drift{\Omega}_\ve(t) = \Big\{x+\frac{b^*t}{\ve}:x\in\Omega_\ve\Big\} .
\end{equation}

\begin{lem}
\label{lem:localization}
Let $(u_\ve,v_\ve)$ be the solution of (\ref{eq:p-1})-(\ref{eq:p-3}). 
Fix a final time $T<+\infty$. 
Then, for any $\delta>0$, there exists $R(\delta)>0$ such that, for any $t\in[0,T]$,
$$
\left\|\drift{u}_\ve(t,x)\right\|_{L^2(\Omega_\ve(t)\cap Q^c_{R(\delta)})}\leq \delta,\hspace{1 cm}
\left\|\drift{v}_\ve(t,x)\right\|_{L^2(\partial\Omega_\ve(t)\cap Q^c_{R(\delta)})}\leq \delta ,
$$
where $Q^c_{R(\delta)}$ is the complement of the cube $Q_{R(\delta)}=]-R(\delta),+R(\delta)[^d$ 
in $\mathbb R^d$.
\end{lem}

\begin{proof}
We rely on an idea of \cite{Marusic05}. 
Let $\phi\in C^\infty (\mathbb R)$ be a smooth cut-off function such that 
$0\le \phi(r)\le 1$, $\phi=0$ for $r\le 1$, $\phi=1$ for $r\ge 2$.
For $x\in\mathbb R^d$, denote $\phi_R(x)=\phi(|x|/R)$. 
Let us consider the variational formulation of (\ref{eq:p-1})-(\ref{eq:p-3}) with 
test functions $(f(u_\ve)\mdrift{\phi}_R,\ve v_\ve\mdrift{\phi}_R)$ where the 
$\mdrift{.}$-notation is defined by (\ref{def.mdrift}). By integration by parts in time, 
the first bulk term is
$$
\itt\id\frac{\partial u_\ve}{\partial t}(s,x) f(u_\ve)(s,x) \mdrift{\phi}_R(s,x) \, {\rm d}x \, {\rm d}s = \frac{1}{\ve}\itt\id F(u_\ve)(s,x)b^*\cdot\nabla\mdrift{\phi}_R(s,x) \, {\rm d}x \, {\rm d}s 
$$
$$
+ \id F(u_\ve)(t,x) \mdrift{\phi}_R(t,x) \, {\rm d}x - \id F(u^{in})(x) \phi_R(x) \, {\rm d}x ,
$$
while, by integration by parts in space, the convective term is
$$
\frac{1}{\ve}\itt\id b_\ve\cdot\nabla u_\ve f(u_\ve) \mdrift{\phi}_R \, {\rm d}x \, {\rm d}s = 
- \frac{1}{\ve}\itt\id F(u_\ve) b_\ve\cdot\nabla \mdrift{\phi}_R \, {\rm d}x \, {\rm d}s ,
$$
and the diffusive term is
$$
-\itt\id\div(D_\ve\nabla u_\ve) f(u_\ve) \mdrift{\phi}_R \, {\rm d}x \, {\rm d}s = 
\itt\id f'(u_\ve) D_\ve\nabla u_\ve \cdot \nabla u_\ve \mdrift{\phi}_R \, {\rm d}x\, {\rm d}s 
$$
$$
+ \itt\id f(u_\ve) D_\ve\nabla u_\ve \cdot \nabla\mdrift{\phi}_R \, {\rm d}x\, {\rm d}s + 
\frac{\kappa}{\ve}\itt\ib(f(u_\ve)-v_\ve) f(u_\ve)\mdrift{\phi}_R \, {\rm d}\sigma(x)\, {\rm d}s .
$$
On the other hand, the boundary terms are
$$
2\ve\dsp\itt\ib \frac{\partial v_\ve}{\partial t} \mdrift{\phi}_R v_\ve \, {\rm d}s \, {\rm d}\sigma(x) 
= \itt\ib b^* \cdot \nabla \mdrift{\phi}_R |v_\ve|^2 \, {\rm d}s \, {\rm d}\sigma(x) 
$$
$$
+ \ve\ib \mdrift{\phi}_R(t,x) |v_\ve(t,x)|^2 \, {\rm d}\sigma(x) 
- \ve\ib \phi_R (x) |v^{in}(x)|^2 \, {\rm d}\sigma(x) ,
$$
$$
2\itt\ib b^s_\ve \cdot \nabla^s {v_\ve}v_\ve \mdrift{\phi}_R \, {\rm d}s \, {\rm d}\sigma(x)
 = - \itt\ib |v_\ve|^2 b^s_\ve \cdot \nabla^s \mdrift{\phi}_R \, {\rm d}s \, {\rm d}\sigma(x) ,
$$
and
$$
-\ve\itt\ib\div^s(D^s_\ve\nabla^s v_\ve) v_\ve \mdrift{\phi}_R \, {\rm d}\sigma(x) \, {\rm d}s = 
\ve\itt\ib D^s_\ve\nabla^s v_\ve \cdot \nabla^s v_\ve \mdrift{\phi}_R \, {\rm d}\sigma(x)\, {\rm d}s 
$$
$$
+ \ve\itt\ib v_\ve  D^s_\ve\nabla^s v_\ve \cdot \nabla^s\mdrift{\phi}_R\, {\rm d}\sigma(x)\, {\rm d}s 
- \frac{\kappa}{\ve}\itt\ib(f(u_\ve)-v_\ve) v_\ve\mdrift{\phi}_R \, {\rm d}\sigma(x)\, {\rm d}s .
$$
Adding these terms together yields:
$$
\id F(u_\ve)(t,x) \mdrift{\phi}_R(t,x) \, {\rm d}x +
\itt\id f'(u_\ve) D_\ve \nabla {u_\ve} \cdot \nabla {u_\ve} \mdrift{\phi}_R \, {\rm d}x \, {\rm d}s
$$
$$
+ \frac\ve 2 \ib |v_\ve(t,x)|^2 \mdrift{\phi}_R(t,x) \, {\rm d}\sigma(x) +
\ve \itt\ib D^s_\ve \nabla^s {v_\ve} \cdot \nabla^s {v_\ve} \mdrift{\phi}_R \, {\rm d}\sigma(x) \, {\rm d}s
$$
$$
+ \frac{\kappa}{\ve} \itt\ib \mdrift{\phi}_R  \left(f(u_\ve) - v_\ve \right)^2  \, {\rm d}\sigma(x) \, {\rm d}s
$$
\begin{equation}
\label{eq:problematic-}
= - \itt\id f(u_\ve) D_\ve\nabla u_\ve \cdot \nabla\mdrift{\phi}_R \, {\rm d}x\, {\rm d}s - \ve\itt\ib v_\ve D^s_\ve\nabla^s v_\ve \cdot \nabla^s\mdrift{\phi}_R \, {\rm d}\sigma(x)\, {\rm d}s
\end{equation}
\begin{equation}
\label{eq:problematic}
+ \frac{1}{\ve}\itt\id F(u_\ve) \left(b_\ve - b^*\right) \cdot \nabla \mdrift{\phi}_R \, {\rm d}x \, {\rm d}s + \itt\ib |v_\ve|^2 \left(b^s_\ve - b^*\right) \cdot \nabla^s \mdrift{\phi}_R \, {\rm d}\sigma(x) \, {\rm d}s
\end{equation}
\begin{equation}
\label{eq:problematic+}
+ \id F(u^{in})(x) \phi_R(x) \, {\rm d}x + \frac{\ve}{2} \ib \phi_R (x) |v^{in}(x)|^2 \, {\rm d}\sigma(x) .
\end{equation}
To arrive at the result, we need to bound the right hand side terms. 
Recall that $f(u_\ve)\leq \alpha u_\ve$. 
By definition of $\phi_R$ we have $\|\nabla\mdrift{\phi}_R\|_{L^\infty(\Omega_\ve)}\leq C/R$, 
so that the first and second terms in (\ref{eq:problematic-}) are bounded by 
\begin{equation}
\label{eq:boundphiR}
\begin{array}{l}
\dsp \frac{C}{R} \Big( \| u_\ve \|_{L^2((0,T)\times\Omega_\ve)} 
\| \nabla u_\ve \|_{L^2((0,T)\times\Omega_\ve)} \\
\dsp \qquad + 
\ve \| v_\ve \|_{L^2((0,T)\times\partial\Omega_\ve)} 
\| \nabla^s v_\ve \|_{L^2((0,T)\times\partial\Omega_\ve)} \Big) 
\leq \frac{C}{R}
\end{array}
\end{equation}
by virtue of Lemma \ref{lem:apriori}.
The two last terms in (\ref{eq:problematic+}), involving the initial data $(u^{in},v^{in})$, 
do not depend on $\ve$ and tend to zero as $R$ tends to $\infty$. 
To cope with the remaining terms in (\ref{eq:problematic}), we introduce two auxiliary problems:
\begin{equation}
\left\{
\begin{array}{ll}
-\Delta\xi_i(y) = b_i^* - b_i(y) & \textrm{in} \: \: Y^0, \\
-\nabla \xi_i\cdot n = 0  & \textrm{on} \: \: \partial \Sigma^0, \\
y \to \xi_i(y) & \textrm{is} \: \:Y-\textrm{periodic,}
\end{array} \right.
\label{eq:aux1}
\end{equation}
\begin{equation}
\left\{
\begin{array}{ll}
-\Delta^s \Xi_i(y) = b_i^* - b^s_i(y) & \textrm{on} \: \: \partial \Sigma^0, \\
y \to \Xi_i(y) & \textrm{is} \: \:Y-\textrm{periodic.}
\end{array} \right.
\label{eq:aux2}
\end{equation}
Both the auxiliary problems admit unique solutions (up to additive constants) since, 
by definition (\ref{eq:drift}) of $b^*$, the source terms in (\ref{eq:aux1}) 
and (\ref{eq:aux2}) are in equilibrium. Substitution of the above auxiliary 
functions in (\ref{eq:problematic}) and integration by parts yields
$$
\sum_{i=1}^d \int_0^t \left( 
\id \ve\nabla \xi_i^\ve \cdot \nabla \Big( F(u_\ve) \partial_{x_i}\mdrift{\phi}_R \Big) \, {\rm d}x + 
\ve\ib \ve\nabla^s\Xi_i^\ve \cdot \nabla^s \Big( |v_\ve|^2 \partial_{x_i}\mdrift{\phi}_R \Big) \, {\rm d}\sigma(x) \right) {\rm d}s.
$$
Since $\ve\nabla\xi_i^\ve(x) = \left( \nabla_y\xi_i\right) (x/\ve)$ and 
$\ve\nabla^s\Xi_i^\ve(x) = \left( \nabla^s_y\Xi_i\right) (x/\ve)$, using 
again the fact that $F(u_\ve)$ has quadratic growth for bounded $u_\ve$, 
the a priori estimates 
from Lemma \ref{lem:apriori} imply that (\ref{eq:problematic}) is bounded by a term 
similar to (\ref{eq:boundphiR}). A final change of the frame of reference and letting 
$R$ go to infinity  leads to the desired result.
\end{proof}

We now prepare the ground for the final compactness result by proving some type 
of equicontinuity in time. 
Let us introduce an orthonormal basis $\{e_j\}_{j\in \mathbb{N}} \in L^2((0,1)^d)$ such that $\{e_j\} \in C^\infty_0([0,1]^d)$. Then the functions $\{e_{jk}\}_{j\in\mathbb{N},k\in\mathbb{Z}^d}$, where $e_{jk}(x)=e_j(x-k)$, form an orthonormal basis in $L^2(\mathbb{R}^d)$.

\begin{lem}
\label{lem:combi}
Let $h>0$ be a small parameter representing time translation. There exists a positive constant $C_{jk}$ independent of $\ve$ and $h$ such that
$$
\Big|\int\limits_{0}^{T-h}\Big\{\int\limits_{\drift{\Omega}_\ve(t+h)}\Big(\drift{u}_\ve + \eta f(\drift{u}_\ve)\Big)(t+h,x)e_{jk}(x)\, {\rm d}x-\int\limits_{\drift{\Omega}_\ve(t)}\Big(\drift{u}_\ve + \eta f(\drift{u}_\ve)\Big)(t,x)e_{jk}(x)\, {\rm d}x\Big\}\, {\rm d}t\Big|
$$
\begin{equation}
\label{eq:asser-combi}
\leq C_{jk}\Big(\sqrt{h}+\ve\Big)
\end{equation}
where $\eta = |\partial\Sigma^0|/|Y^0|$.
\end{lem}

\begin{proof}
We compute the difference 
$$
\Big(\drift{u}_\ve(t+h,x),e_{jk}(x)\Big)_{L^2(\drift{\Omega}_\ve(t+h))} - \Big(\drift{u}_\ve(t,x),e_{jk}(x)\Big)_{L^2(\drift{\Omega}_\ve(t))}
$$
$$
+\ve\Big(\drift{v}_\ve(t+h,x),e_{jk}(x)\Big)_{L^2(\partial\drift{\Omega}_\ve(t+h))} -\ve\Big(\drift{v}_\ve(t,x),e_{jk}(x)\Big)_{L^2(\partial\drift{\Omega}_\ve(t))}
$$
$$
= \int_{t}^{t+h}\frac{{\rm d}}{{\rm d}s}\Big\{\int_{\drift{\Omega}_\ve(s)}\drift{u}_\ve(s,x)e_{jk}(x)\, {\rm d}x+\ve\int_{\partial\drift{\Omega}_\ve(s)}\drift{v}_\ve(s,x)e_{jk}(x)\, {\rm d}\sigma(x)\Big\}\, {\rm d}s
$$
$$
= \int_{t}^{t+h}\frac{{\rm d}}{{\rm d}s}\Big\{\int_{\Omega_\ve}u_\ve(s,x)\mdrift{e}_{jk}(x)\, {\rm d}x+\ve\int_{\partial\Omega_\ve}v_\ve(s,x)\mdrift{e}_{jk}(x)\, {\rm d}\sigma(x)\Big\}\, {\rm d}s
$$
$$
= \int_{t}^{t+h}\int_{\Omega_\ve}\Big\{\frac{\partial u_\ve}{\partial s}(s,x)\mdrift{e}_{jk}(x) - \frac{b^*}{\ve}\cdot\nabla \mdrift{e}_{jk}(x) u_\ve(s,x)\Big\}\, {\rm d}x\, {\rm d}s
$$
$$
+\int_{t}^{t+h}\ve\int_{\partial\Omega_\ve}\Big\{\frac{\partial v_\ve}{\partial s}(s,x)\mdrift{e}_{jk}(x) - \frac{b^*}{\ve}\cdot\nabla \mdrift{e}_{jk}(x) v_\ve(s,x)\Big\}\, {\rm d}\sigma(x)\, {\rm d}s
$$
$$
= \frac{1}{\ve}\int_{t}^{t+h}\int_{\Omega_\ve}\Big\{\Big(b_\ve-b^*\Big)\cdot\nabla \mdrift{e}_{jk}(x)u_\ve(s,x)- D_\ve\nabla u_\ve(s,x)\cdot\nabla \mdrift{e}_{jk}(x)\Big\}\, {\rm d}x\, {\rm d}s
$$
$$
+\ve\int_{t}^{t+h}\int_{\partial\Omega_\ve}\Big\{\frac{1}{\ve}\Big(b^s_\ve-b^*\Big)\cdot\nabla \mdrift{e}_{jk}(x)v_\ve(s,x)- D^s_\ve\nabla^s v_\ve(s,x)\cdot\nabla^s \mdrift{e}_{jk}(x)\Big\}\, {\rm d}\sigma(x)\, {\rm d}s
$$
$$
= \int_{t}^{t+h}\int_{\Omega_\ve}\Big\{\ve\Delta\xi^\ve_i(x)\partial_{x_i}\mdrift{e}_{jk}(x)u_\ve(s,x) - D_\ve\nabla u_\ve(s,x)\cdot\nabla \mdrift{e}_{jk}(x)\Big\}\, {\rm d}x\, {\rm d}s
$$
$$
+\int_{t}^{t+h}\int_{\partial\Omega_\ve}\Big\{\ve^2\Delta^s\Xi^\ve_i(x)\partial_{x_i}\mdrift{e}_{jk}(x)v_\ve(s,x)- \ve D^s_\ve\nabla^s v_\ve(s,x)\cdot\nabla^s \mdrift{e}_{jk}(x)\Big\}\, {\rm d}\sigma(x)\, {\rm d}s
$$
$$
= -\int_{t}^{t+h}\int_{\Omega_\ve}\Big\{\nabla_y\xi^\ve_i(x)\cdot\nabla\Big(\partial_{x_i}\mdrift{e}_{jk}(x)u_\ve(s,x)\Big) + D_\ve\nabla u_\ve(s,x)\cdot\nabla \mdrift{e}_{jk}(x)\Big\}\, {\rm d}x\, {\rm d}s
$$
$$
-\ve\int_{t}^{t+h}\int_{\partial\Omega_\ve}\Big\{\nabla^s_y\Xi^\ve_i(x)\cdot\nabla^s\Big(\partial_{x_i}\mdrift{e}_{jk}(x)v_\ve(s,x)\Big)+ D^s_\ve\nabla^s v_\ve(s,x)\cdot\nabla^s \mdrift{e}_{jk}(x)\Big\}\, {\rm d}\sigma(x)\, {\rm d}s
$$
$$
\leq C_{jk}\sqrt{h} .
$$
The above bound follows from the a priori estimates (\ref{eq:apriori}). By the definition of $w_\ve$, we have $\drift{v}_\ve = f(\drift{u}_\ve) - \ve\drift{w}_\ve$. Substituting for $\drift{v}_\ve$ in the above inequality yields
$$
\Big|\Big(\drift{u}_\ve(t+h,x),e_{jk}(x)\Big)_{L^2(\drift{\Omega}_\ve(t+h))} + \ve\Big(f(\drift{u}_\ve)(t+h,x),e_{jk}(x)\Big)_{L^2(\partial\drift{\Omega}_\ve(t+h))}
$$
$$
- \Big(\drift{u}_\ve(t,x),e_{jk}(x)\Big)_{L^2(\drift{\Omega}_\ve(t))} - \ve\Big(f(\drift{u}_\ve)(t,x),e_{jk}(x)\Big)_{L^2(\partial\drift{\Omega}_\ve(t))}\Big|
$$
$$
\leq C_{jk}\sqrt{h} + \Big|\ve^2\Big(\drift{w}_\ve(t+h,x),e_{jk}(x)\Big)_{L^2(\partial\drift{\Omega}_\ve(t+h))}\Big| +  \Big|\ve^2\Big(\drift{w}_\ve(t,x),e_{jk}(x)\Big)_{L^2(\partial\drift{\Omega}_\ve(t))}\Big| .
$$
We now replace the boundary integrals involving the nonlinear term with volume integrals. 
To that end, we introduce an auxiliary problem:
\begin{equation}
\left\{ \begin{array}{ll}
  \div_y \Upsilon(y) = \eta = |\partial\Sigma^0|/Y^0 & \textrm{in} \: \: Y^0,\\[0.1cm]
  \Upsilon \cdot n = 1 & \textrm{on} \: \: \partial\Sigma^0,\\[0.1cm]
  y \to \Upsilon(y) & \textrm{is} \: \:Y-\textrm{periodic,}
 \end{array} \right.
\label{eq:auxinequ}
\end{equation}
which admits a smooth $Y$-periodic vector solution $\Upsilon$. Then,
$$
\ve\int_{\partial\drift{\Omega}_\ve(t)}f(\drift{u}_\ve)(t,x)e_{jk}(x)\, {\rm d}\sigma(x) = \ve\ib f(u_\ve)(t,x)\mdrift{e}_{jk}(x)\, {\rm d}\sigma(x)
$$
$$
= \ve\int_{\partial\Omega_\ve} f(u_\ve)(t,x)\mdrift{e}_{jk}(x)\Big(\Upsilon^\ve(x)\cdot n\Big)\, {\rm d}\sigma(x) = \ve\int_{\Omega_\ve} \div\Big(f(u_\ve)(t,x)\mdrift{e}_{jk}(x)\Upsilon^\ve(x)\Big)\, {\rm d}x
$$
$$
=\ve\int_{\Omega_\ve} f'(u_\ve)(t,x)\nabla u_\ve(t,x)\cdot\Upsilon^\ve(x)\mdrift{e}_{jk}(x)\, {\rm d}x +\ve\int_{\Omega_\ve}\nabla \mdrift{e}_{jk}(x)\cdot\Upsilon^\ve(x)f(u_\ve)(t,x)\, {\rm d}x
$$
$$
+\eta\int_{\Omega_\ve}\mdrift{e}_{jk}(x)f(u_\ve)(t,x)\, {\rm d}x .
$$
The above calculation leads to
$$
\Big|\int_{t}^{t+h}\frac{{\rm d}}{{\rm d}s}\int_{\drift{\Omega}_\ve(s)}\Big(\drift{u}_\ve + \eta f(\drift{u}_\ve)\Big)(s,x)e_{jk}(x)\, {\rm d}x\, {\rm d}s\Big|
$$
$$
\leq C_{jk}\sqrt{h}
$$
$$
+ \Big|\ve^2\ib w_\ve(t+h,x)e_{jk}\Big(x-\frac{b^*(t+h)}{\ve}\Big)\, {\rm d}\sigma(x)\Big| + \Big|\ve^2\ib w_\ve(t,x)e_{jk}\Big(x-\frac{b^*t}{\ve}\Big)\, {\rm d}\sigma(x)\Big|
$$
$$
+\Big|\ve\int_{\Omega_\ve} f'(u_\ve)(t,x)\nabla u_\ve(t,x)\cdot\Upsilon^\ve(x)e_{jk}\Big(x-\frac{b^*t}{\ve}\Big)\, {\rm d}x\Big|
$$
$$
+\Big|\ve\int_{\Omega_\ve}\nabla e_{jk}\Big(x-\frac{b^*t}{\ve}\Big)\cdot\Upsilon^\ve(x)f(u_\ve)(t,x)\, {\rm d}x\Big|
$$
$$
+\Big|\ve\int_{\Omega_\ve} f'(u_\ve)(t,x)\nabla u_\ve(t,x)\cdot\Upsilon^\ve(x)e_{jk}\Big(x-\frac{b^*(t+h)}{\ve}\Big)\, {\rm d}x\Big|
$$
$$
+\Big|\ve\int_{\Omega_\ve}\nabla e_{jk}\Big(x-\frac{b^*(t+h)}{\ve}\Big)\cdot\Upsilon^\ve(x)f(u_\ve)(t,x)\, {\rm d}x\Big| .
$$
We integrate the above inequality over $(0,T-h)$. As $0\leq f(u_\ve)\leq \alpha u_\ve$ and $0\leq f'(u_\ve)\leq \alpha$, the a priori estimates in (\ref{eq:apriori}) lead to (with a possibly different constant 
$C_{jk}$)
$$
\int_{0}^{T-h}\left|\int_{t}^{t+h}\frac{{\rm d}}{{\rm d}s}\int_{\drift{\Omega}_\ve(s)}\Big(\drift{u}_\ve + \eta f(\drift{u}_\ve)\Big)(s,x)e_{jk}(x)\, {\rm d}x\, {\rm d}s\right|dt\leq C_{jk}\Big(\sqrt{h}+\ve\Big)
$$
which is nothing but (\ref{eq:asser-combi}).
\end{proof}

To prove the compactness of $u_\ve$, an intermediate result is to prove the 
compactness of the sequence $z_\ve$ defined by 
$$
z_\ve(t,x)=u_\ve(t,x)+\eta f(u_\ve)(t,x) \quad \mbox{ for } \quad (t,x)\in(0,T)\times\Omega_\ve . 
$$
In view of (\ref{eq:apriori}), $z_\ve$ satisfies the following estimates 
\begin{equation}
\label{eq:z-est}
\begin{array}{ll}
\dsp\id|z_\ve|^2\, {\rm d}x \leq \id (1+\eta\alpha)^2 |u_\ve|^2\, {\rm d}x \leq C 
\quad \forall\, t \in(0,T),\\[0.3 cm]
\dsp\int_{0}^{T}\id|\nabla z_\ve|^2\, {\rm d}x = \int_{0}^{T}\id(1+\eta f'(u_\ve))^2|\nabla u_\ve|^2\, {\rm d}x\leq C.
\end{array}
\end{equation}
We recall from \cite{CS:79, Acerbi92} that there exists an extension operator 
$E_\ve:H^1(\Omega_\ve)\to H^1(\R^d)$ which satisfies the following property: 
there exists a constant $C$, independent of $\ve$, such that, 
for any function $\phi_\ve\in H^1(\Omega_\ve)$, $E_\ve \phi_\ve\Big|_{\Omega_\ve} = \phi_\ve$ and
\begin{equation}
\label{eq:prop-ext}
\|E_\ve \phi_\ve\|_{L^2(\R^d)}\leq C\|\phi_\ve\|_{L^2(\Omega_\ve)},\:\:\:\:
\|\nabla E_\ve \phi_\ve\|_{L^2(\R^d)}\leq C\|\nabla \phi_\ve\|_{L^2(\Omega_\ve)} .
\end{equation}
As we are proving compactness in moving coordinates, we consider the sequences 
$\drift{z}_\ve$ and $\drift{E_\ve z_\ve}$ where the $\drift{.}$-operator is 
defined by (\ref{def.drift}). 
The decomposition of these two functions in terms of the orthonormal basis $\{e_{jk}\}$ of $L^2(\R^d)$ yields
$$
\drift{z}_\ve(t,x) = \sum_{j\in\mathbb{N} , k\in\mathbb{Z}^d} \mu^\ve_{jk}(t) e_{jk}(x) \quad \mbox{ with } \quad 
\mu^\ve_{jk}(t) = \int_{\drift{\Omega}_\ve(t)}\drift{z}_\ve(t,x)e_{jk}(x)\, {\rm d}x ,
$$
$$
\drift{E_\ve z_\ve}(t,x) = \sum_{j\in\mathbb{N} , k\in\mathbb{Z}^d} \nu^\ve_{jk}(t) e_{jk}(x) \quad \mbox{ with } \quad 
\nu^\ve_{jk}(t) = \itr\drift{E_\ve z_\ve}(t,x)e_{jk}(x)\, {\rm d}x ,
$$
where $\mu^\ve_{jk}(t)$ and $\nu^\ve_{jk}(t)$ are the time dependent Fourier coefficients.

\begin{lem}
\label{lem:RFK}
There exists a subsequence, still denoted by $\ve$, such that
$$
\mu^\ve_{jk}\to\mu_{jk}\hspace{1 cm}\textrm{in}\:\:L^2(0,T)\:\hspace{0.5 cm}\:\forall j\in\mathbb{N},k\in\mathbb{Z}^d ,
$$
for some $\mu_{jk}\in L^2(0,T)$. Further, the function
$$
z_0(t,x) = \dsp\sum_{j\in\mathbb{N} , k\in\mathbb{Z}^d} \mu_{jk}(t) e_{jk}(x)
$$
is an element of $L^2((0,T)\times\R^d)$.
\end{lem}

\begin{proof}
From Lemma \ref{lem:combi}, we have
\begin{equation}
\label{eq:transl}
\int_{0}^{T-h}\Big|\mu^\ve_{jk}(t+h) - \mu^\ve_{jk}(t)\Big|\, {\rm d}t\leq C_{jk}\Big(\sqrt{h}+\ve\Big).
\end{equation}
Inequality (\ref{eq:transl}) is a variant of the Riesz-Fr\'echet-Kolmogorov criterion for (strong) 
compactness in $L^1(0,T)$ (see e.g. \cite{Brezis83}, page 72, Theorem IV.25), the variant being 
caused by the additional $\ve$-term in the right hand side. It is not difficult to check that 
the proof of compactness is still valid with this additional term (see \cite{Harsha} if necessary). 
Therefore, for any $j\in\mathbb{N},k\in\mathbb{Z}^d$, there is a subsequence $\ve_{jk}\to0$ and a limit 
$\mu_{jk}\in L^1(0,T)$ such that
$$
\mu^{\ve_{jk}}_{jk}\to\mu_{jk}\hspace{1 cm}\textrm{in}\:\:L^1(0,T) .
$$
A diagonalization procedure yields another subsequence $\ve$ such that the above convergence in $L^1$ holds for all indices $j,k$. The a priori estimates (\ref{eq:apriori}) on $u_\ve$ in turn implies that the Fourier coefficients are bounded in $L^\infty(0,T)$ too. Thus, the above strong compactness property is true in every $L^p$, $1\leq p<+\infty$, and, in particular, in $L^2(0,T)$. The assertion that $z_0\in L^2((0,T)\times\R^d)$ follows from the observation that
$$
\dsp\Big\|z_0\Big\|^2_{L^2((0,T)\times\R^d)} = \sum_{j\in\mathbb{N} , k\in\mathbb{Z}^d} \iti|\mu_{jk}(t)|^2\, {\rm d}t
$$
$$
\leq \sum_{j\in\mathbb{N} , k\in\mathbb{Z}^d} \liminf_{\ve\to0}\iti|\mu^{\ve}_{jk}|^2\, {\rm d}t \leq \liminf_{\ve\to0}\Big\|\drift{z}_\ve\Big\|^2_{L^2((0,T)\times\drift{\Omega}_\ve(t))}<\infty .
$$
\end{proof}

The next result states that there is not much difference between the time Fourier coefficients 
of $\drift{z}_\ve$ (defined in the perforated domain $\drift{\Omega}_\ve(t)$) and of its 
extension $\drift{E_\ve z_\ve}$. 

\begin{lem}
\label{lem:Fou-diff}
Let $\theta=|Y^0|/|Y|\in(0,1)$. There exists a constant $C_{jk}$ independent of $\ve$ such that
\begin{equation}
\label{eq:Fou-diff}
\Big|\mu^\ve_{jk}(t) - \theta \nu^\ve_{jk}(t)\Big|\leq C_{jk}\ve .
\end{equation}
\end{lem}

\begin{proof}
By definition of the Fourier coefficients, we have
$$
\mu^\ve_{jk}(t) - \theta \nu^\ve_{jk}(t) = \int_{\drift{\Omega}_\ve(t)} \drift{z}_\ve(t,x)e_{jk}(x)\, {\rm d}x - \theta \itr\drift{E_\ve z_\ve}(t,x)e_{jk}(x)\, {\rm d}x
$$
\begin{equation}
\label{eq:1-Fou-diff}
= \itr E_\ve z_\ve(t,x)\mdrift{e}_{jk}(x)\Big(\chi(x/\ve)-\theta\Big)\, {\rm d}x ,
\end{equation}
where $\chi(x/\ve)$ is the characteristic function of $\Omega_\ve$, or equivalently 
$\chi(y)$ is the characteristic function of $Y^0$.
Let us introduce the following auxiliary problem:
\begin{equation}
\left\{ \begin{array}{ll}
 - \div_y( \nabla_y\Phi(y)) = \chi(y) - \theta & \textrm{in} \: \: Y,\\[0.1cm]
y\to\Phi(y) & \textrm{is}\:\:Y\mbox{-periodic.}
 \end{array} \right.
\label{eq:aux-Fou-diff}
\end{equation}
Using (\ref{eq:aux-Fou-diff}) in (\ref{eq:1-Fou-diff}) leads to
$$
\Big|\mu^\ve_{jk}(t) - \theta \nu^\ve_{jk}(t)\Big|\leq \ve\itr\Big|\nabla_y\Phi(x/\ve)\cdot\nabla\Big( E_\ve z_\ve(t,x)\mdrift{e}_{jk}(x)\Big)\Big|\, {\rm d}x .
$$
The properties (\ref{eq:prop-ext}) of the extension operator $E_\ve$ and the estimates (\ref{eq:z-est}) lead to (\ref{eq:Fou-diff}).
\end{proof}

A last technical result is the possibility of truncating the modal series (with respect to $j$) 
of a sequence which is bounded in $L^2((0,T);H^1(\R^d))$.

\begin{lem}
\label{lem:truncate}
Let $\phi_\ve(t,x)$ be a bounded sequence in $L^2((0,T);H^1(\R^d))$. 
For any $\delta>0$, there exists a $J(\delta)$ such that for all $\ve$ we have
\begin{equation}
\label{eq:lem-trunc}
\Big\|\phi_\ve\chi_{Q_{R(\delta)}} - \sum_{|k|\leq R(\delta) , \, |j|\leq J(\delta)}\lambda^\ve_{jk}(t) e_{jk}(x)\Big\|_{L^2((0,T)\times\R^d)}\leq\delta ,
\end{equation}
where $Q_{R(\delta)}$ is the cube defined in Lemma \ref{lem:localization} and $\lambda^\ve_{jk}(t)$ are the time dependent Fourier coefficients of $\phi_\ve$ defined as
$$
\lambda^\ve_{jk}(t)=\itr\phi_\ve(t,x)e_{jk}(x)\, {\rm d}x.
$$
\end{lem}

\begin{proof}
As $Q_{R(\delta)}$ is a bounded domain, the expansion of $\phi_\ve\chi_{Q_{R(\delta)}}$ in the basis $\{e_{jk}\}$ can be truncated in $k$ with $|k|\leq R(\delta)$ and is still exact. Let us consider the unit ball 
$$
B = \{v\in H^1(Q_{R(\delta)}):\|v\|_{H^1(Q_{R(\delta)})}\leq1\}.
$$
We know that $H^1(Q_{R(\delta)})$ is pre-compact in $L^2(Q_{R(\delta)})$ \cite{Brezis83}. 
Hence for a given $\delta>0$ and for all $v\in B$, there exists $J(\delta)$ such that
$$
\dsp\Big\|\sum_{|k|\leq R(\delta) , \, j> J(\delta)}(v,e_{jk})_{L^2(Q_{R(\delta)})} e_{jk}\Big\|^2_{L^2(Q_{R(\delta)})}\leq\delta
$$
Now, given $\phi_\ve\in L^2((0,T);H^1(Q_{R(\delta)}))$, we have $\phi_\ve(t)\in H^1(Q_{R(\delta)})$ for almost every $t\in (0,T)$. Thus for a given $\delta>0$, there exists a $J(\delta)$ such that
$$
\dsp\Big\|\sum_{|k|\leq R(\delta) , \, j> J(\delta)}(\phi_\ve(t),e_{jk})_{L^2(Q_{R(\delta)})} e_{jk}\Big\|^2_{L^2(Q_{R(\delta)})}\leq\delta\|\phi_\ve(t)\|^2_{H^1(Q_{R(\delta)})}
$$
for almost every $t\in (0,T)$. Integrating the above expression over $(0,T)$, we arrive at
$$
\dsp\Big\|\sum_{|k|\leq R(\delta) , \, j> J(\delta)}(\phi_\ve(t),e_{jk})_{L^2((0,T)\times Q_{R(\delta)})} e_{jk}\Big\|^2_{L^2(Q_{R(\delta)})}\leq\delta\|\phi_\ve(t)\|^2_{L^2((0,T);H^1(Q_{R(\delta)}))} ,
$$
which implies the result (\ref{eq:lem-trunc}).
\end{proof}

We are now ready to state the compactness result of the sequence $z_\ve$. Note that the limit 
is not $z_0$ but $z_0/\theta$ since $z_0$ was the limit of the sequence $z_\ve$ extended 
by zero outside the porous domain $\drift{\Omega}_\ve(t)$. 

\begin{thm}
\label{thm:compact}
There exists a subsequence $\ve$ such that
\begin{equation}
\label{eq:thm-comp}
\lim_{\ve\to0}\iti\int_{\drift{\Omega}_\ve(t)}|\drift{z}_\ve(t,x) - \theta^{-1} z_0(t,x)|^2\, {\rm d}x\, {\rm d}t = 0.
\end{equation}
\end{thm}

\begin{proof}
The estimates (\ref{eq:z-est}) for $\{\drift{z}_\ve\}$, being similar to (\ref{eq:apriori}), 
imply that the localization principle, Lemma \ref{lem:localization}, holds true for the 
sequence $\{\drift{z}_\ve\}$ too. Thus, for a given $\delta>0$, there exists a $R(\delta)>0$ 
big enough such that
\begin{equation}
\label{eq:telscop-1}
\Big\|\drift{z}_\ve - \drift{z}_\ve \chi_{Q_{R(\delta)}}\Big\|_{L^2((0,T)\times\drift{\Omega}_\ve(t))}\leq\frac{\delta}{5} .
\end{equation}
Applying Lemma \ref{lem:truncate} to $\drift{E_\ve z_\ve}\chi_{Q_{R(\delta)}}$, for any $\delta > 0$, there exists $J(\delta)$ such that, for any small $\ve>0$,
\begin{equation}
\label{eq:toprove-2}
\Big\|\drift{E_\ve z_\ve}\chi_{Q_{R(\delta)}} - \sum_{|k|\leq R(\delta) , \, |j|\leq J(\delta)}\nu^\ve_{jk}(t) e_{jk}(x)\Big\|_{L^2((0,T)\times\R^d)}\leq\frac{\delta}{5} .
\end{equation}
As $\drift{E_\ve z_\ve}$ is an extension of $\drift{z}_\ve$, we deduce from (\ref{eq:toprove-2}) that
\begin{equation}
\label{eq:telscop-2}
\Big\|\drift{z}_\ve\chi_{Q_{R(\delta)}} - \sum_{|k|\leq R(\delta) , \, |j|\leq J(\delta)}\nu^\ve_{jk}(t) e_{jk}(x)\Big\|_{L^2((0,T)\times\drift{\Omega}_\ve(t))}\leq\frac{\delta}{5} .
\end{equation}
From Lemma \ref{lem:Fou-diff}, for a given $\delta>0$ and $\ve$ small enough, we have
\begin{equation}
\label{eq:telscop-3}
\Big\|\sum_{|k|\leq R(\delta) , \, |j|\leq J(\delta)}\nu^\ve_{jk}(t) e_{jk}(x) - \frac{1}{\theta}\sum_{|k|\leq R(\delta) , \, |j|\leq J(\delta)}\mu^\ve_{jk}(t) e_{jk}(x)\Big\|_{L^2((0,T)\times\drift{\Omega}_\ve(t))}\leq\frac{\delta}{5} .
\end{equation}
Lemma \ref{lem:RFK} asserted that the Fourier coefficients are relatively compact in $L^2(0,T)$. Thus, for $\ve$ small enough, we have
\begin{equation}
\label{eq:telscop-4}
\Big\|\sum_{|k|\leq R(\delta) , \, |j|\leq J(\delta)}\mu^\ve_{jk}(t) e_{jk}(x) - \sum_{|k|\leq R(\delta) , \, |j|\leq J(\delta)}\mu_{jk}(t) e_{jk}(x)\Big\|_{L^2((0,T)\times\drift{\Omega}_\ve(t))}\leq \theta\frac{\delta}{5} .
\end{equation}
By Lemma \ref{lem:RFK} we know that $z_0\in L^2((0,T)\times\R^d)$ so, by choosing a large enough $J(\delta)$, we have
\begin{equation}
\label{eq:telscop-5}
\Big\|\sum_{|k|\leq R(\delta) , \, |j|\leq J(\delta)}\mu_{jk}(t) e_{jk}(x) - z_0(t,x)\Big\|_{L^2((0,T)\times Q_{R(\delta)})}\leq\frac{\delta}{5} .
\end{equation}
Thus summing up (\ref{eq:telscop-1}), (\ref{eq:telscop-2}), (\ref{eq:telscop-3}), (\ref{eq:telscop-4}) and (\ref{eq:telscop-5}) we arrive at
\begin{equation}
\label{eq:rel-comp}
\Big\|\drift{z}_\ve(t,x) - \theta^{-1} z_0(t,x)\Big\|_{L^2((0,T)\times\drift{\Omega}_\ve(t))}\leq\delta
\end{equation}
which is (\ref{eq:thm-comp}).
\end{proof}
 
Eventually, we deduce the desired compactness of the sequence $u_\ve$ 
from that of $z_\ve$.

\begin{cor}
\label{cor:comp-u}
There exists a subsequence $\ve$ and a limit $u_0\in L^2((0,T)\times\R^d)$ such that
\begin{equation}
\label{eq:cor-comp}
\lim_{\ve\to0}\iti\int_{\drift{\Omega}_\ve(t)}|\drift{u}_\ve(t,x) - u_0(t,x)|^2\, {\rm d}x\, {\rm d}t = 0.
\end{equation}
\end{cor}

\begin{proof}
Since the nonlinear isotherm $f$ is bounded and monotone, the application $(I+\eta f)$ is globally 
invertible with linear growth. We have $u_\ve(t,x) = (I+\eta f)^{-1}z_\ve(t,x)$ and the compactness 
property of $\{\drift{z}_\ve\}$, as stated in Theorem \ref{thm:compact}, immediately 
translates to $\{\drift{u}_\ve\}$ by a standard application of the Lebesgue dominated convergence 
theorem.
\end{proof}

\begin{rem}
\label{rem:ext}
Compactness results are crucial in the homogenization of nonlinear parabolic equations. 
Another approach sharing some similarities with us can be found in \cite{APP} where the 
authors rely on the extension operator of \cite{CS:79, Acerbi92}. However, it seems 
difficult to adapt this approach in the present context for at least two reasons. First, 
one of the unknown, the surface concentration $v_\ve$, is defined merely on the solid 
boundary $\partial\Omega_\ve$ so that it requires a specific type of extension to the whole 
space $\mathbb{R}^d$. Second, it is not at all obvious to show that each extension of $u_\ve$ and 
of $v_\ve$ satisfy the time equicontinuity of Lemma \ref{lem:combi}.
\end{rem}

\section{Two-scale convergence with drift}
\label{sec:2sc}

The goal of this section is to derive the homogenized problem corresponding to 
the original system (\ref{eq:p-1})-(\ref{eq:p-3}). More precisely we shall prove 
a weak convergence of the sequence of solutions $(u_\ve, v_\ve)$ to the homogenized 
solution, in the sense of two-scale convergence with drift (a notion introduced in 
\cite{Marusic05}, see \cite{Allaire08} for detailed proofs). Section \ref{sec:strong} 
will provide a strong convergence result for $(u_\ve, v_\ve)$ under additional 
assumptions. We start by recalling the notion of two-scale convergence with drift, 
which is a generalization of the usual two-scale convergence \cite{Allaire92, 
Nguetseng89}. Let us remark that this rigorous two-scale convergence with 
drift corresponds to the, simpler albeit heuristic, method of two-scale asymptotic 
expansions with drift, as described in \cite{Allaire10}. In the sequel, the subscript 
$\#$ denotes spaces of $Y$-periodic functions. 

\begin{prop}\cite{Marusic05}
\label{compact}
Let $\mathcal V$ be a constant vector in $\mathbb{R}^d$. For any bounded 
sequence of functions $U_\ve(t,x) \in L^2((0,T)\times\mathbb{R}^d)$, 
i.e., satisfying 
$$
\|U_\ve\|_{L^2((0,T)\times\mathbb{R}^d)} \leq C ,
$$
there exists a limit $U_0(t,x,y) \in L^2((0,T)\times\mathbb{R}^d\times Y)$ 
and one can extract a subsequence (still denoted by $\ve$) which is said 
to two-scale converge with drift $\mathcal V$, or equivalently 
in moving coordinates $(t,x)\rightarrow(t,x-\frac{{\mathcal V} t}{\ve})$, 
to this limit, in the sense that, for any 
$\phi(t,x,y) \in C_0^\infty((0,T)\times\mathbb{R}^d ; C_\#^\infty(Y) )$, 
$$
\dsp\lim_{\ve\to0}\displaystyle\int_0^T\int_{\mathbb{R}^d} U_\ve(t,x) \phi(t,x-\frac{{\mathcal V} t}{\ve},\frac{x}{\ve})\, {\rm d}x \, {\rm d}t = \displaystyle\int_0^T\int_{\mathbb{R}^d}\int_{\mathbb{T}^d}U_0(t,x,y)\phi(t,x,y)\, {\rm d}y \, {\rm d}x \, {\rm d}t .
$$
We denote this convergence by $U_\ve \ts U_0$.
\end{prop}

In the sequel we shall apply Proposition \ref{compact} with the drift 
${\mathcal V} = b^*$. 

\begin{rem}
Proposition \ref{compact} equally applies to a sequence $u_\ve(t,x) \in L^2((0,T)\times\Omega_\ve)$, merely defined in the perforated domain $\Omega_\ve$, and satisfying the uniform bound
$$
\|u_\ve\|_{L^2((0,T)\times\Omega_\ve)} \leq C .
$$
In such a case we obtain 
$$
\lim_{\ve\to0}\int_0^T\int_{\Omega_\ve} u_\ve(t,x) \phi(t,x-\frac{{\mathcal V} t}{\ve},\frac{x}{\ve})\, {\rm d}x \, {\rm d}t = \int_0^T\int_{\mathbb{R}^d}\int_{Y^0}U_0(t,x,y)\phi(t,x,y)\, {\rm d}y \, {\rm d}x \, {\rm d}t .
$$
\end{rem}

Proposition \ref{compact} can be generalized in several ways as follows 
(the proofs are standard, see \cite{Allaire08} if necessary). 
In particular, following the lead of \cite{Allaire96, Neuss96}, it 
can be extended to sequences defined on periodic surfaces. 

\begin{prop}
\label{h1-domain-conv}
Let ${\mathcal V} \in \mathbb{R}^d$ and let the sequence $U_\ve$ be uniformly bounded in 
$L^2((0,T);H^1(\mathbb{R}^d))$. 
Then, there exist a subsequence, still denoted by $\ve$, and functions 
$U_0(t,x) \in L^2((0,T);H^1(\mathbb{R}^d))$ and $U_1(t,x,y) \in
 L^2((0,T)\times\mathbb{R}^d;H^1_\#(Y))$ such that
$$
U_\ve \ts U_0
\quad \mbox{ and } \quad 
\nabla U_\ve \ts \nabla_x U_0 + \nabla_y U_1 .
$$
\end{prop}

\begin{prop}
\label{compact-bdry}
Let ${\mathcal V} \in \mathbb{R}^d$ and let $W_\ve$ be a sequence in 
$L^2((0,T)\times\partial\Omega_\ve)$ such that
$$
\ve \int_0^T\int_{\partial\Omega_\ve} |W_\ve(t,x)|^2 \, {\rm d}\sigma(x) \, {\rm d}t \leq C.
$$ 
Then, there exist a subsequence, still denoted by $\ve$, 
and a function $W_0(t,x,y) \in L^2((0,T)\times\mathbb{R}^d\times\partial\Sigma^0)$ 
such that $W_\ve(t,x)$ two-scale converges with drift $\mathcal V$ to $W_0(t,x,y)$ 
in the sense that
$$
\lim_{\ve\to0}\ve \int_0^T \int_{\partial\Omega_\ve} W_\ve(t,x) \phi(t,x-\frac{{\mathcal V} t}{\ve},\frac{x}{\ve})\, {\rm d}\sigma(x) \, {\rm d}t 
$$
$$
= \displaystyle\int_0^T\int_{\mathbb{R}^d}\int_{\partial\Sigma^0} W_0(t,x,y) \phi(t,x,y) \, {\rm d}\sigma(y) \, {\rm d}x \, {\rm d}t
$$
for any $\phi(t,x,y) \in C_0^\infty((0,T)\times\mathbb{R}^d; C_\#^\infty(Y))$. 
We denote this convergence by $W_\ve \tss W_0$.
\end{prop}

\begin{prop}
\label{h1-bdry-conv}
Let $W_\ve(t,x) \in L^2((0,T);H^1(\partial\Omega_\ve))$ be such that
$$
\ve \displaystyle \int_0^T \int_{\partial\Omega_\ve} 
\left( |W_\ve(t,x)|^2 + |\nabla^s W_\ve(t,x)|^2 \right) d\sigma(x) \, {\rm d}t \leq C .
$$ 
Then, there exist a subsequence, still denoted by $\ve$, and functions 
$W_0(t,x) \in L^2((0,T);H^1(\mathbb{R}^d))$ and 
$W_1(t,x,y) \in L^2((0,T)\times\mathbb{R}^d;H^1_\#(\partial\Sigma^0))$
 such that
$$
W_\ve \tss W_0(t,x)
$$
$$
\nabla^s W_\ve \tss G(y)\nabla_x W_0(t,x) + \nabla^s_y W_1(t,x,y)
$$
where $G(y)$ is the projection operator on the tangent plane of 
$\partial\Sigma^0$ at point $y$.
\end{prop}

Eventually we state a technical lemma which will play a key role in the convergence 
analysis.

\begin{lem}
\label{technical}
Let $\phi(t,x,y) \in L^2((0,T)\times\mathbb{R}^d\times\partial\Sigma^0)$ be such that 
$\int\limits_{\partial\Sigma^0} \phi(t,x,y) \, {\rm d}\sigma(y) = 0$ for a.e.  
$(t,x)\in(0,T)\times \mathbb{R}^d$. There exist two periodic vector fields 
$\theta(t,x,y) \in L^2((0,T)\times\mathbb{R}^d\times Y^0)^d$ and 
$\Theta(t,x,y) \in L^2((0,T)\times\mathbb{R}^d\times\partial\Sigma^0)^d$ 
such that
\begin{equation}
\left\{ \begin{array}{ll}
  \div_y \theta = 0 & \textrm{in} \: \: Y^0,\\[0.1cm]
  \theta \cdot n = \phi & \textrm{on} \: \: \partial\Sigma^0, \\[0.1cm]
  \div^s_y \Theta = \phi & \textrm{on} \: \: \partial\Sigma^0.
 \end{array} \right.
\label{cruclem}
\end{equation}
\end{lem}

\begin{proof}
We choose $\theta = \nabla_y \xi$ with $\xi \in H^1_\#(Y^0)$ a solution to
\begin{equation}
\left\{ \begin{array}{ll}
  \Delta_y \xi = 0 & \textrm{in} \: \: Y^0,\\
  \nabla_y \xi \cdot n = \phi & \textrm{on} \: \: \partial\Sigma^0,
   \end{array} \right.
\label{cruclem1}
\end{equation}
which admits a unique solution, up to an additive constant, since the 
compatibility condition of (\ref{cruclem1}) is satisfied. 
On similar lines, we choose $\Theta = \nabla^s_y \beta$ where $\beta$ 
is the unique solution in $H^1_\#(\partial\Sigma^0)/\mathbb{R}$
of $\Delta^S_y \beta =\phi$ on $\partial\Sigma^0$ which is solvable 
because of the zero-average assumption on $\phi$. 
\end{proof}

We now apply the above results on two-scale convergence with drift 
to the homogenization of (\ref{eq:p-1})-(\ref{eq:p-3}) to deduce our 
main result.

\begin{thm}
\label{main-weak}
Under assumption (\ref{eq:drift}) which defines a common average value 
$b^*$ for the bulk and surface velocities, the sequence of bulk and 
surface concentrations $(u_\ve,v_\ve)$, solutions of system 
(\ref{eq:p-1})-(\ref{eq:p-3}), two-scale converge with drift $b^*$, 
as $\ve \to 0$, in the following sense:
\begin{equation}
\label{eq:tscv}
\left\{
\begin{array}{ll}
u_\ve \ts u_0(t,x)\\
v_\ve \tss f(u_0)(t,x)\\
\nabla u_\ve \ts\nabla_x u_0(t,x) + \nabla_y [\chi(y) \cdot \nabla_x u_0(t,x)]\\
\nabla^s v_\ve \tss f'(u_0)\Big[G(y)\nabla_x u_0(t,x) + \nabla^s_y (\omega(y) \cdot \nabla_x u_0(t,x))\Big]\\
\frac{1}{\ve} \left(f(u_\ve) - v_\ve\right) \tss f'(u_0)\Big[\chi(y)-\omega(y)\Big] \cdot \nabla_x u_0(t,x)
\end{array} \right.
\end{equation}
where $u_0(t,x)$ is the unique solution of the homogenized problem:
\begin{equation}
\label{eq:hom}
\left\{
\begin{array}{ll}
\dsp \left[|Y^0| + |\partial\Sigma^0| f'(u_0) \right] \frac{\partial u_0}{\partial t} 
- \div_x(A^*(u_0)\nx u_0) = 0 & \textrm{in } (0,T)\times\mathbb{R}^d,\\[0.4cm] 
\dsp \left[|Y^0|u_0 + |\partial\Sigma^0| f(u_0) \right](0,x) = |Y^0|u^{in}(x) + |\partial\Sigma^0| v^{in}(x) & \textrm{in } \mathbb{R}^d,
\end{array} \right.
\end{equation}
the dispersion tensor $A^*$ is given by its entries:
\begin{equation}
\label{exp-disp}
\begin{array}{ll}
\dsp A_{ij}^*(u_0) & \dsp = \iy D \left( \ny\chi_i + e_i\right) \cdot \left(\ny\chi_j + e_j\right) \, {\rm d}y\\[0.5cm]
& \dsp + \kappa f'(u_0) \ip \left(\chi_i - \omega_i\right) \left(\chi_j - \omega_j\right) \, {\rm d}\sigma(y)\\[0.5cm]
& \dsp + \dsp f'(u_0) \ip D^s \left( \ny^s \omega_i+e_i\right) \cdot 
\left(\ny^s\omega_j + e_j\right) \, {\rm d}\sigma(y)\\[0.3 cm]
& + \dsp \iy D(y)\Big(\ny\chi_j\cdot e_i - \ny\chi_i\cdot e_j\Big)\, {\rm d}y\\[0.3 cm]
& +\dsp\frac{\alpha}{(1+\beta u_0)^2} \ip D^s(y)\Big(\ny^s\omega_j\cdot e_i - \ny^s\omega_i\cdot e_j\Big)\, {\rm d}\sigma(y)\\[0.3 cm]
& \dsp+ \iy \Big(b(y)\cdot\ny\chi_i\Big)\chi_j\, {\rm d}y +\frac{\alpha}{(1+\beta u_0)^2} \ip \Big(b^s(y)\cdot\ny^s\omega_i\Big)\omega_j\, {\rm d}\sigma(y),
\end{array}
\end{equation}
with $(\chi,\omega) = (\chi_i,\omega_i)_{1\leq i \leq d}$ being the solution of the cell problem:
\begin{equation}
\left\{ 
\begin{array}{lll}
-b^* \cdot e_i + b(y)\cdot(e_i + \ny \chi_i) - \div_y(D(e_i + \ny \chi_i)) = 0 & \textrm{in}\:\: Y^0,\\[0.3cm]
- D(y)\left( e_i + \ny \chi_i\right)\cdot n = \dsp\kappa f'(u_0) \left(\chi_i - \omega_i \right) & \textrm{on}\:\:\partial \Sigma^0, \\[0.3cm]
-b^* \cdot e_i + b^s(y)\cdot(e_i + \ny^s \omega_i) - \div^s_y(D^s(e_i+ \nabla^s_y \omega_i)) = \kappa \left(\chi_i -  \omega_i \right) & \textrm{on}\:\:\partial \Sigma^0, \\[0.3cm]
y \to (\chi_i(y),\omega_i(y)) \hspace{6 cm} Y-\textrm{periodic.}
 \end{array}\right.
\label{eq:cellpb}
\end{equation}
\end{thm}

\begin{rem}
\label{rem.strong}
Theorem \ref{main-weak} gives a weak type convergence result for the sequences 
$u_\ve$ and $v_\ve$ since two-scale convergence with drift relies on the use of 
test functions. However, by virtue of Corollary \ref{cor:comp-u} the convergence 
is strong for $u_\ve$ in the sense that 
$$
\lim_{\ve\to0} \left\|u_\ve(t,x) - u_0\left(t, x-\frac{b^*}{\ve}t\right)\right\|_{L^2((0,T)\times \Omega_\ve)} = 0 .
$$
A similar strong convergence result for $v_\ve$ and for their gradients 
will be proven in Section \ref{sec:strong}. 
\end{rem}

Before proving Theorem \ref{main-weak} we establish the well-posed 
character of the homogenized and cell problems.

\begin{lem}
\label{lem.uniq}
For any given value of $u_0(t,x)\geq0$, the cell problem (\ref{eq:cellpb}) admits 
a unique solution $(\chi_i,\omega_i) \in H^1_\#(Y^0) \times H^1_\#(\partial\Sigma^0)$, 
up to the addition of a constant vector $(C,C)$ with $C\in\mathbb{R}$. 

The homogenized problem (\ref{eq:hom}) admits a unique solution 
$u_0 \in C([0,T] ; L^2(\mathbb{R}^d))$ and $\nabla u_0 \in L^2((0,T)\times\mathbb{R}^d)$. 
\end{lem}

\begin{proof}
The variational formulation of (\ref{eq:cellpb}) is
$$
\iy \Big(b(y)\cdot\ny \chi_i\Big)\varphi\, {\rm d}y + \iy D(e_i + \ny \chi_i)\cdot\ny\varphi\, {\rm d}y + f'(u_0)\ip \Big(b^s(y)\cdot\ny^s \omega_i\Big)\psi\, {\rm d}\sigma(y)
$$
$$
+f'(u_0)\ip D^s(e_i+ \nabla^s_y \omega_i)\cdot\ny^s\psi\, {\rm d}\sigma(y) + \kappa f'(u_0)\ip(\chi_i -  \omega_i)(\varphi - \psi)\, {\rm d}\sigma(y)
$$
\begin{equation}
\label{eq:vf-cpb}
= \iy\Big(b^* - b(y)\Big)\cdot e_i\varphi\, {\rm d}y + f'(u_0)\ip\Big(b^* - b^s(y)\Big)\cdot e_i\psi\, {\rm d}\sigma(y)
\end{equation}
to which the Lax-Milgram lemma can be easily applied. 

The symmetric part of $A^*$ is given by
\begin{equation}
\label{exp-disp-int}
\begin{array}{ll}
\dsp A_{ij}^{*\mbox{sym}}(u_0) & \dsp = \iy D \left( \ny\chi_i + e_i\right) \cdot \left(\ny\chi_j + e_j\right) \, {\rm d}y\\[0.5cm]
& \dsp + \kappa f'(u_0) \ip \left(\chi_i - \omega_i\right) \left(\chi_j - \omega_j\right) \, {\rm d}\sigma(y)\\[0.5cm]
& \dsp + \dsp f'(u_0) \ip D^s \left( \ny^s \omega_i+e_i\right) \cdot 
\left(\ny^s\omega_j + e_j\right) \, {\rm d}\sigma(y).
\end{array}
\end{equation}
Since $f'(u_0)\geq0$, (\ref{exp-disp-int}) implies that $A^*(u_0) \geq \int_{Y^0} D(y) \, {\rm d}y$ 
and thus the dispersion tensor is uniformly coercive. On the other hand, since 
$f'(u_0)\leq\alpha$, $A^*(u_0)$ is uniformly bounded from above. Then, it is a 
standard process to prove existence and uniqueness of (\ref{eq:hom}) 
(see \cite{Lady68} if necessary). 
\end{proof}

\begin{proof}[Proof of Theorem \ref{main-weak}] 
Lemma \ref{lem:apriori} furnishes a priori estimates so that, 
up to a subsequence, all sequences in (\ref{eq:tscv}) have two-scale limits 
with drift, thanks to the previous Propositions \ref{h1-domain-conv}, 
\ref{compact-bdry} and \ref{h1-bdry-conv}. The first task is to identify 
those limits. Similar computations were performed in \cite{AllaireMikelic10}, so we 
content ourselves in explaining how to derive the limit of the most delicate term, 
that is $w_\ve = \frac{1}{\ve} \left(f(u_\ve) - v_\ve\right)$, assuming that the 
other limits are already characterized. As opposed to \cite{AllaireMikelic10}, 
where only linear terms were involved, we have to identify the weak two-scale drift 
limit of $f(u_\ve)$. In view of Corollary \ref{cor:comp-u} which states the
compactness of $\drift{u}_\ve(t,x)$, and since $f(u)\leq\alpha u$,  
it is easily deduced that $f(u_\ve)$ two-scale converges with drift to $f(u_0)$. 

Let us denote by $q(t,x,y)$ the two-scale drift limit of $w_\ve$ and 
let us choose a test function $\phi$ as in Lemma \ref{technical}, i.e., 
$\int\limits_{\partial\Sigma^0} \phi(t,x,y) \, {\rm d}\sigma(y) = 0$. 
By definition
$$
\lim_{\ve\to0} \ve \displaystyle\int_0^T\ib w_\ve(t,x) \phi\left(t,x-\frac{b^* t}{\ve},\frac{x}{\ve}\right) \, {\rm d}\sigma(x) \, {\rm d}t = 
\int_0^T\int_{\mathbb{R}^d}\int_{\partial\Sigma^0} q \phi \, {\rm d}\sigma(y) \, {\rm d}x \, {\rm d}t .
$$
Replacing $w_\ve$ by the difference between $f(u_\ve)$ and $v_\ve$ we get 
a different two-scale limit which will allows us to characterize $q(t,x,y)$. 
In view of (\ref{cruclem}), we first have
$$
\ve \int_0^T\ib \frac{1}{\ve} f(u_\ve) \phi\left(t,x-\frac{b^* t}{\ve},\frac{x}{\ve}\right) \, {\rm d}\sigma(x) \, {\rm d}t
$$
$$
= \int_0^T\id \div\left(f(u_\ve) \theta\left(t,x-\frac{b^* t}{\ve},\frac{x}{\ve}\right)\right) \, {\rm d}x \, {\rm d}t ,
$$
$$
= \int_0^T\id\left[f'(u_\ve)\nabla u_\ve \cdot \theta\left(t,x-\frac{b^* t}{\ve},\frac{x}{\ve}\right) + f(u_\ve) \left(\div_x \theta\right)\left(t,x-\frac{b^* t}{\ve},\frac{x}{\ve}\right)\right] \, {\rm d}x \, {\rm d}t ,
$$
which, using again the compactness of Corollary \ref{cor:comp-u}, converges, as $\ve$ goes to 0, to
$$ 
\int_0^T\int_{\mathbb{R}^d}\int_{Y^0} f'(u_0)\Big[\left(\nabla_x u_0 + \nabla_y u_1\right) \cdot \theta + f(u_0) \div_x \theta\Big] \, {\rm d}y \, {\rm d}x \, {\rm d}t
$$
$$
= \int_0^T\int_{\mathbb{R}^d}\int_{\partial\Sigma^0} f'(u_0)u_1 \theta \cdot n \, {\rm d}\sigma(y) \, {\rm d}x \, {\rm d}t = \displaystyle\int_0^T\int_{\mathbb{R}^d}\int_{\partial\Sigma^0} f'(u_0)u_1 \phi \, {\rm d}\sigma(y) \, {\rm d}x \, {\rm d}t .
$$
On the other hand, the second term is
$$
\ve \int_0^T\ib \frac{1}{\ve} v_\ve \phi\left(t,x-\frac{b^* t}{\ve},\frac{x}{\ve}\right) \, {\rm d}\sigma(x) \, {\rm d}t = \displaystyle\int_0^T\ib v_\ve \left(\div^s_y \Theta\right)\left(t,x-\frac{b^* t}{\ve},\frac{x}{\ve}\right) \, {\rm d}\sigma(x) \, {\rm d}t
$$
$$
=\ve \int_0^T\ib v_\ve \left[\div^s\left(\Theta\left(t,x-\frac{b^* t}{\ve},\frac{x}{\ve}\right)\right) - \div_x \left(G \Theta\right)\left(t,x-\frac{b^* t}{\ve},\frac{x}{\ve}\right)\right]\, {\rm d}\sigma(x) \, {\rm d}t
$$
$$
= - \ve \int_0^T\ib \left[ \Theta\left(t,x-\frac{b^* t}{\ve},\frac{x}{\ve}\right) \cdot \nabla^s v_\ve + \div_x \left(G \Theta\right)\left(t,x-\frac{b^* t}{\ve},\frac{x}{\ve}\right) v_\ve\right]\, {\rm d}\sigma(x) \, {\rm d}t
$$
which converges, as $\ve$ goes to 0, to
$$
- \int_0^T\int_{\mathbb{R}^d}\int_{\partial\Sigma^0} \left[ \Theta \cdot \left(G(y)\nabla_x f(u_0) + \nabla^s_y v_1\right) + \div_x\left(G(y)\Theta\right) f(u_0)\right]\, {\rm d}\sigma(y) \, {\rm d}x \, {\rm d}t
$$
$$
= \int_0^T\int_{\mathbb{R}^d}\int_{\partial\Sigma^0} v_1 \div^s_y \Theta \, {\rm d}\sigma(y) \, {\rm d}x \, {\rm d}t =  \displaystyle\int_0^T\int_{\mathbb{R}^d}\int_{\partial\Sigma^0} v_1 \phi \, {\rm d}\sigma(y) \, {\rm d}x \, {\rm d}t .
$$
Subtracting the two limit terms, we have shown that
$$
\int_0^T\int_{\mathbb{R}^d}\int_{\partial\Sigma^0} q \phi \, {\rm d}\sigma(y) \, {\rm d}x \, {\rm d}t = \displaystyle\int_0^T\int_{\mathbb{R}^d}\int_{\partial\Sigma^0} \left(f'(u_0)u_1-v_1\right) \phi \, {\rm d}\sigma(y) \, {\rm d}x \, {\rm d}t ,
$$ 
for all $\phi$ such that $\int\limits_{\partial\Sigma^0} \phi \, {\rm d}y = 0$.
Thus,
$$
q(t,x,y) = f'(u_0)(t,x)u_1(t,x,y)  - v_1(t,x,y) + l(t,x)
$$
for some function $l(t,x)$ which does not depend on $y$. 
Since, $u_1$ and $v_1$ are also defined up to 
the addition of a function solely dependent on $(t,x)$, we can get rid of $l(t,x)$ 
and we recover indeed the last line of (\ref{eq:tscv}).

The rest of the proof is now devoted to show that $u_0(t,x)$ is the solution of 
the homogenized equation (\ref{eq:hom}). For that goal, we shall pass to the 
limit in the coupled variational formulation of (\ref{eq:p-1})-(\ref{eq:p-3}),
\begin{equation}
\label{eq:fv}
\int_0^T\id\left[\frac{\partial u_\ve}{\partial t} \phi_\ve + \frac{1}{\ve} b_\ve \cdot \nabla u_\ve \phi_\ve + D_\ve \nabla u_\ve \cdot 
\nabla \phi_\ve\right] \, {\rm d}x \, {\rm d}t 
\end{equation}
$$
\hspace{1.7 cm}+ \ve \int_0^T\ib\left[\frac{\partial v_\ve}{\partial t} \psi_\ve + \frac{1}{\ve} b^s_\ve \cdot \nabla^s v_\ve \psi_\ve +D_\ve^s \nabla^s v_\ve \cdot 
\nabla^s \psi_\ve \right] \, {\rm d}\sigma(x) \, {\rm d}t
$$
$$
+ \frac{\kappa}{\ve}\int_0^T\ib\left[ \left(f(u_\ve) - v_\ve\right) \left(\phi_\ve - \psi_\ve\right) \right] \, {\rm d}\sigma(x) \, {\rm d}t= 0 ,
$$
with the test functions 
$$
\phi_\ve = \phi\left(t,x-\frac{b^* t}{\ve}\right) + \ve \phi_1\left(t,x-\frac{b^* t}{\ve}, \frac{x}{\ve}\right) ,
$$
$$
\psi_\ve = \phi\left(t,x-\frac{b^* t}{\ve}\right) + \ve \psi_1\left(t,x-\frac{b^* t}{\ve}, \frac{x}{\ve}\right) .
$$
Here $\phi(t,x)$, $\phi_1(t,x,y)$ and $\psi_1(t,x,y)$ are smooth compactly supported 
functions which vanish at $t=T$. Let us consider the convective terms in (\ref{eq:fv}) and 
perform integration by parts:
$$
\int_0^T\id \left(\frac{\partial u_\ve}{\partial t} + \frac{1}{\ve} b_\ve \cdot \nabla u_\ve \right) \phi_\ve \, {\rm d}x \, {\rm d}t 
+ \ve \int_0^T\ib \left( \frac{\partial v_\ve}{\partial t} + \frac{1}{\ve} b^s_\ve \cdot \nabla^s v_\ve\right) \psi_\ve \, {\rm d}\sigma(x) \, {\rm d}t
$$
$$
= -\int_0^T\id u_\ve \frac{\partial\phi}{\partial t}\drfg \, {\rm d}x \, {\rm d}t 
+ \frac{1}{\ve}\int_0^T\id u_\ve b^* \cdot \nabla_x \phi \drfg \, {\rm d}x \, {\rm d}t
$$
$$
+\int_0^T\id u_\ve b^* \cdot \nabla_x \phi_1 \drfyg \, {\rm d}x \, {\rm d}t - \id u^{in}(x) \phi(0,x) \, {\rm d}x + \mathcal O(\ve)
$$
$$
-\frac{1}{\ve}\int_0^T\id u_\ve b_\ve \cdot \nabla_x\phi\drfg\, {\rm d}x\, {\rm d}t 
+ \int_0^T\id b_\ve\ \cdot \nabla u_\ve \phi_1 \drfyg \, {\rm d}x \, {\rm d}t
$$
$$
-\ve \int_0^T\ib v_\ve \frac{\partial\phi}{\partial t}\drfg \, {\rm d}\sigma(x) \, {\rm d}t 
+ \int_0^T\ib v_\ve b^* \cdot \nabla_x\phi\drfg \, {\rm d}\sigma(x) \, {\rm d}t
$$
$$
+ \ve \int_0^T \ib v_\ve b^* \cdot \nabla_x \psi_1\drfyg \, {\rm d}\sigma(x) \, {\rm d}t 
- \ve \ib v^{in}(x) \phi(0,x) \, {\rm d}\sigma(x) + \mathcal O(\ve)
$$
$$
-\int_0^T\ib v_\ve b^s_\ve \cdot \nabla_x\phi\drfg \, {\rm d}\sigma(x) \, {\rm d}t 
+ \ve\int_0^T\ib b^s_\ve \cdot \nabla^s v_\ve \psi_1 \drfyg \, {\rm d}\sigma(x) \, {\rm d}t .
$$
We cannot directly pass to the two-scale limit since there are terms which 
apparently are of order $\ve^{-1}$. We thus regroup them and, recalling 
definition (\ref{def.mdrift}) of the transported function $\mdrift{\phi}$ 
and using the two auxiliary problems (\ref{eq:aux1}) and (\ref{eq:aux2}), 
we deduce
$$
\int_0^T\id u_\ve \frac{b^* - b_\ve}{\ve}\cdot\nabla_x \mdrift{\phi} \, {\rm d}x \, {\rm d}t 
+ \int_0^T\ib v_\ve \left(b^*-b^s_\ve\right) \cdot \nabla_x\mdrift{\phi} \, {\rm d}\sigma(x) \, {\rm d}t
$$
$$
= \ve\sum_{i=1}^d \int_0^T\id u_\ve \Delta\xi_i^\ve 
\partial_{x_i}\mdrift{\phi} \, {\rm d}x \, {\rm d}t + 
\ve^2\sum_{i=1}^d \int_0^T\ib v_\ve \Delta^s\Xi_i^\ve 
\partial_{x_i}\mdrift{\phi} \, {\rm d}\sigma(x) \, {\rm d}t 
$$
$$
= - \ve\sum_{i=1}^d \int_0^T\id \nabla \xi_i^\ve \cdot \nabla \Big( u_\ve \partial_{x_i}\mdrift{\phi} \Big) \, {\rm d}x \, {\rm d}t -
\ve^2\sum_{i=1}^d \int_0^T\ib \nabla^s\Xi_i^\ve \cdot \nabla \Big( v_\ve \partial_{x_i}\mdrift{\phi} \Big) \, {\rm d}\sigma(x) \, {\rm d}t ,
$$
for which we can pass to the two-scale limit.

In a first step, we choose $\phi\equiv0$ and we pass to the two-scale limit with drift 
in (\ref{eq:fv}). It yields
$$
\int_0^T\itr\iy u_0(t,x) \, b^* \cdot \nabla_x \phi_1 (t,x,y) \, {\rm d}y \, {\rm d}x \, {\rm d}t 
$$
$$
+ \int_0^T\itr\iy b(y) \cdot \left(\nabla_x u_0(t,x) + \nabla_y u_1(t,x,y)\right) \phi_1(t,x,y)\, {\rm d}y \, {\rm d}x \, {\rm d}t
$$
$$
+ \int_0^T\itr\iy D(y) \left(\nabla_x u_0(t,x) + \nabla_y u_1(t,x,y)\right) 
\cdot \nabla_y \phi_1(t,x,y) \, {\rm d}y \, {\rm d}x \, {\rm d}t
$$
$$
+ \int_0^T\itr\ip f(u_0) \, b^* \cdot \nabla_x \psi_1(t,x,y) \, {\rm d}\sigma(y) \, {\rm d}x \, {\rm d}t 
$$
$$
+ \int_0^T\itr \ip b^s(y) \cdot \left(f'(u_0) G\nabla_x u_0(t,x) + \nabla^s_y v_1(t,x,y)\right) \psi_1(t,x,y)\, {\rm d}\sigma(y) \, {\rm d}x \, {\rm d}t
$$
$$
+ \int_0^T\itr\ip D^s(y)\left(f'(u_0) G\nabla_x u_0(t,x) + \nabla^s_y v_1(t,x,y)\right)\cdot \nabla^s_y \psi_1(t,x,y)\, {\rm d}\sigma(y)\, {\rm d}x\, {\rm d}t
$$
$$
+ \kappa\int_0^T\itr\ip \left(f'(u_0)u_1 - v_1\right)\left(\phi_1 - \psi_1\right) \, {\rm d}\sigma(y)\, {\rm d}x\, {\rm d}t=0 ,
$$
which is nothing but the variational formulation of 
\begin{equation}
\left\{ 
\begin{array}{ll}
\dsp - b^* \cdot \nx u_0 + b\cdot(\nx u_0 + \ny u_1) - \div_y(D(\nx u_0 + \ny u_1)) = 0 
& \textrm{in } Y^0,\\[0.3cm]
\dsp - D\left( \nx u_0 + \ny u_1\right)\cdot n = \frac{\kappa}{(1+\beta u_0)^2}\left(\alpha u_1 - (1+\beta u_0)^2v_1 \right)  & \textrm{on}\:\:\partial \Sigma^0, \\[0.3cm]
\dsp - b^* \cdot \alpha\nx u_0 + b(y)\cdot(\alpha\nx u_0 + (1+\beta u_0)^2 \ny u_1) \\[0.3cm]
- \div^s_y(D^s(\alpha\nabla_x u_0 + (1+\beta u_0)^2\nabla^s_y v_1)) = \kappa \left(\alpha u_1 - (1+\beta u_0)^2 v_1\right) & \textrm{on}\:\:\partial \Sigma^0, \\[0.3cm]
y \to (u_1(y),v_1(y)) \quad Y-\textrm{periodic,} &
 \end{array}\right.
\label{eq:e-1}
\end{equation}
which implies that $u_1 = \chi(y) \cdot \nabla_x u_0$ and 
$v_1 = f'(u_0) \omega(y) \cdot \nabla_x u_0$ where $(\chi,\omega)$ 
is the solution of the cell problem (\ref{eq:cellpb}).

In a second step, we choose $\phi_1\equiv 0$, $\psi_1\equiv 0$ and 
we pass to the two-scale limit with drift in (\ref{eq:fv}). It yields
$$
|Y^0|\itr \frac{\partial u_0}{\partial t} \phi \, {\rm d}x \, {\rm d}t
+|\partial\Sigma^0|\itr f'(u_0)\frac{\partial u_0}{\partial t} \phi\, {\rm d}x \, {\rm d}t
$$
$$
-|Y^0|\int_{\mathbb{R}^d} u^{in}(x) \phi(0,x) \, {\rm d}x
-|\partial\Sigma^0|\int_{\mathbb{R}^d} v^{in}(x) \phi(0,x) \, {\rm d}x
$$
$$
+ \sum_{i,j=1}^d\:\itr\iy D_{ij}(y) \frac{\partial u_0}{\partial x_j}\frac{\partial \phi}{\partial x_i} \, {\rm d}x\, {\rm d}t
+ \itr\iy \sum_{i,j=1}^d\sum_{l=1}^d D_{il}(y) \frac{\partial\chi_j(y)}{\partial y_l}\frac{\partial u_0}{\partial x_j}\frac{\partial \phi}{\partial x_i} \, {\rm d}x\, {\rm d}t
$$
$$
+\itr\ip \sum_{i,j=1}^d\sum_{l=1}^d f'(u_0) D^s_{il}(y) G_{lj}(y) \frac{\partial u_0}{ \partial x_j}\frac{\partial\phi}{\partial x_i} \, {\rm d}\sigma(y)\, {\rm d}x\, {\rm d}t
$$
$$
+\itr\ip \sum_{i,j=1}^d\sum_{l=1}^d f'(u_0) D^s_{il}(y) \frac{\partial u_0}{\partial x_j}\frac{\partial^s\omega_j(y)}{\partial y_l}\frac{\partial\phi}{\partial x_i} \, {\rm d}\sigma(y)\, {\rm d}x\, {\rm d}t
$$
$$
+\itr\iy \sum_{i,j=1}^d \sum_{l=1}^d \frac{\partial\xi_i(y)}{\partial y_l} \frac{\partial\chi_j(y)}{\partial y_l}\frac{\partial u_0}{\partial{x_j}}\frac{\partial\phi}{\partial x_i}\, {\rm d}y\, {\rm d}x\, {\rm d}t
$$
$$
+\itr\ip \sum_{i,j=1}^d \sum_{l=1}^d f'(u_0) \frac{\partial^s\Xi_i(y)}{\partial y_l} \frac{\partial^s\omega_j(y)}{\partial y_l}\frac{\partial u_0}{\partial{x_j}}\frac{\partial\phi}{\partial x_i}\, {\rm d}\sigma(y)\, {\rm d}x\, {\rm d}t = 0 ,
$$
which is precisely the variational formulation of the homogenized problem (\ref{eq:hom}) 
where $A^*$ is given by
$$
A^*_{ij}(u_0) = \iy D e_i\cdot e_j\, {\rm d}y + \iy D\ny\chi_j\cdot e_i\, {\rm d}y
$$
$$
 + f'(u_0)\ip D^s e_i\cdot e_j\, {\rm d}\sigma(y) + f'(u_0)\ip D^s\ny^s\omega_j\cdot e_i\, {\rm d}\sigma(y)
$$
\begin{equation}
\label{eq:disp-sym}
+ \iy \ny\xi_i\cdot\ny\chi_j\, {\rm d}y +  f'(u_0)\ip\ny^s\Xi_i\cdot\ny^s\omega_j\, {\rm d}\sigma(y) .
\end{equation}
It remains to prove that formula (\ref{eq:disp-sym}) is equivalent to that announced in (\ref{exp-disp}). 
The two auxiliary problems (\ref{eq:aux1}) and (\ref{eq:aux2}) have to be used for transforming 
(\ref{eq:disp-sym}) into (\ref{exp-disp}). Let us test (\ref{eq:aux1}) for $\xi_i$ by the cell solution $\chi_j$ followed by testing  (\ref{eq:aux2}) for $\Xi_i$ by $f'(u_0)\omega_j$. Adding the thus obtained expressions leads to
$$
\iy\ny\xi_i\cdot\ny\chi_j\, {\rm d}y + f'(u_0)\ip\ny^s\Xi_i\cdot\ny^s\omega_j\, {\rm d}\sigma(y)
$$
\begin{equation}
\label{eq:disp-aux}
= \iy\left(b^*_i - b_i\right)\chi_j\, {\rm d}y + f'(u_0)\ip\left(b^*_i - b^s_i\right)\omega_j\, {\rm d}\sigma(y).
\end{equation}
Next, in the variational formulation (\ref{eq:vf-cpb}) for $(\chi_i,\omega_i)$, we shall replace the test functions by $(\chi_j,\omega_j)$:
$$
\iy\left(b^*_i - b_i(y)\right)\chi_j(y)\, {\rm d}y + f'(u_0)\ip\left(b^*_i - b^s_i(y)\right)\omega_j(y)\, {\rm d}\sigma(y)
$$
$$
= \iy D(y) \ny\chi_i\cdot \ny\chi_j \, {\rm d}y + f'(u_0) \ip D^s(y)\ny^s \omega_i\cdot \ny^s\omega_j \, {\rm d}\sigma(y)
$$
$$
+ \iy D\ny\chi_j\cdot e_i\, {\rm d}y + f'(u_0)\ip D^s\ny^s\omega_j\cdot e_i\, {\rm d}\sigma(y)
$$
\begin{equation}
\label{eq:disp-cpb}
+ \kappa f'(u_0)\ip \left[\chi_i - \omega_i\right]\left[\chi_j - \omega_j\right] \, {\rm d}\sigma(y) .
\end{equation}
Finally, using (\ref{eq:disp-aux}) and (\ref{eq:disp-cpb}) in (\ref{eq:disp-sym}) shows that 
both formulas (\ref{eq:disp-sym}) and (\ref{exp-disp}) for the dispersion tensor $A^*$ are 
equivalent. 
Therefore, we have indeed obtained the variational formulation of the 
homogenized problem (\ref{eq:hom}) which, by Lemma \ref{lem.uniq}, admits a unique solution. 
As a consequence of uniqueness, the entire sequence converges, not merely a subsequence.
\end{proof}

\begin{rem}
\label{rem:bxy}
We are assuming that the velocity fields are purely periodic functions, depending only on the fast variable $y=x/\varepsilon$ and not on the slow variable $x$. In particular, we are unable to treat the case of more general locally periodic velocity fields of the type $b_\varepsilon(x) = b(x,x/\varepsilon)$ and $b^s_\varepsilon(x) = b^s(x,x/\varepsilon)$ where $b(x,y)$ and $b^s(x,y)$ are smooth divergence-free, with respect to both variables, vector fields. The main technical reason is that the homogenized drift $b^*$ would then depend on $x$ which cannot be handled by our method. We are lacking the adequate tools (even formal ones) to guess the correct effective limit. Even more, we know from \cite{AO} that, under special assumptions on the coefficients depending on $x$ and $y$, a new localization phenomenon can happen which is completely different from the asymptotic behavior proved in the present work. 

We are also assuming condition (\ref{eq:drift}) of equal bulk and surface drifts. 
If it is not satisfied, we don't know how to homogenize the nonlinear problem although 
the linear case is well understood \cite{AllaireHutridurga}. 
\end{rem}

\begin{rem}
\label{rem:dxy:local}
In Remark \ref{rem:bxy}, we noticed that our approach cannot handle locally periodic velocity fields. However, if the diffusion tensors
are of such type, i.e., $D_\ve(x)=D(x,x/\ve)$, $D^s_\ve(x)=D^s(x,x/\ve)$,
then it does not change the definition of the drift $b^*$ which
still makes the cell problem (\ref{eq:cellpb}) well-posed.
Our above analysis can be carried out for locally periodic
diffusion tensors with only
minor modifications. In particular, some derivatives with respect to $x$ of $D$ and $D^s$ appear in the homogenized equation (\ref{eq:hom}).
\end{rem}

\begin{rem}
\label{rem:freundlich}
According to the literature (see e.g. \cite{vDK92}), there are two kinds of concave isotherms - Langmuir and Freundlich. A function $f(u)$ is said to be of Langmuir type if it is strictly concave near $u = 0$ and $f'(0+) < +\infty$. On the other hand, $f(u)$ is said to be of Freundlich type if it is strictly concave near $u = 0$ and $f'(0+) = +\infty$. An example of one such isotherm is $f(u) = K\: u^p$, $0 < p < 1$ and $K>0$ an equilibrium constant. In the case of a Freundlich isotherm, a formal analysis, based on two-scale 
asymptotic expansions with drift, would yield the same results, namely homogenized and cell problems, 
as in Theorem \ref{main-weak}. Even with $f'(0+) = +\infty$ these results have a meaning (in particular, 
for $u_0=0$ it forces the equality $\chi_i=\omega_i$). However we are unable to rigorously prove the 
convergence of the homogenization process. 
\end{rem}

\section{Strong Convergence}
\label{sec:strong}

Theorem \ref{main-weak} gives a weak type convergence result for the sequences 
$u_\ve$ and $v_\ve$ in the sense of two-scale convergence with drift. Thanks 
to the strong compactness of Corollary \ref{cor:comp-u} it was immediately 
improved in Remark \ref{rem.strong} as a strong convergence result for $u_\ve$ 
in the $L^2$-norm. In the present section, we recover this result and additionally 
prove the strong convergence of $v_\ve$ and of their gradients, up to the addition 
of some corrector terms. 
The main idea is to show that the energy associated with (\ref{eq:p-1})-(\ref{eq:p-3}) 
converges to that of the homogenized equation (\ref{eq:hom}). This is shown to work under a 
specific constraint on the initial data $(u^{in},v^{in})$ which must be well prepared 
(see below). Then, our argument relies on the notion of strong two-scale convergence 
which is recalled in Lemma \ref{prop:strocv}. 
Theorem \ref{thm:strong} is the main result of the section. Following ideas of \cite{AllaireMikelic10, AllaireHutridurga}, its proof relies on the lower semicontinuity property of the norms 
with respect to the (weak) two-scale convergence. The additional difficulty is the 
nonlinear terms which arise in the energy equality (\ref{eq:energy}). 
Lemma \ref{lem:convex} is a technical result of strong two-scale convergence 
adapted to our nonlinear setting. 

Let us explain the assumption on the well prepared character of the initial data 
and its origin. 
We denote by $u^0_0(x)$ the initial data of the homogenized problem (\ref{eq:hom}), 
which is defined on $\mathbb{R}^d$ by
\begin{equation}
\label{eq:comp2}
|Y^0|u^0_0 + |\partial\Sigma^0|f(u^0_0) = |Y^0|u^{in} + |\partial\Sigma^0|v^{in} .
\end{equation}
Since $f$ is non negative and increasing, (\ref{eq:comp2}) uniquely defines $u^0_0$ 
as a nonlinear function of $(u^{in},v^{in})$. 
It will turn out that, passing to the limit in the energy equality, and thus 
deducing strong convergence, requires another constraint for $u^0_0$ which is
\begin{equation}
\label{eq:comp1}
|Y^0| F(u^0_0) + \frac 12 |\partial\Sigma^0| f^2(u^0_0) = 
|Y^0| F(u^{in}) + \frac12 |\partial\Sigma^0| (v^{in})^2 ,
\end{equation}
where $F$ is the primitive of $f$. In general, (\ref{eq:comp2}) and (\ref{eq:comp1}) 
are not compatible, except if the initial data $(u^{in},v^{in})$ satisfies a 
compatibility condition which, for the moment, we admittedly write as a nonlinear 
relationship:
\begin{equation}
\label{eq:comp3}
{\mathcal H}(u^{in},v^{in})=0 . 
\end{equation}
Typically, if $u^{in}$ is known, (\ref{eq:comp3}) prescribes a given value 
for $v^{in}$. 
Lemma \ref{lem:comp} will investigate the existence and uniqueness of a solution 
$v^{in}$ in terms of given $u^{in}$. In the linear case, namely $f(u)=\alpha u$, 
(\ref{eq:comp3}) reduces to the explicit relationship $v^{in} = \alpha u^{in}$.

\begin{thm}
\label{thm:strong}
Let the initial data $(u^{in},v^{in})$ satisfy the nonlinear equation (\ref{eq:comp3}). 
Then the sequences $u_\ve(t,x)$ and $v_\ve(t,x)$ strongly two-scale converge 
with drift in the sense that 
\begin{equation}
\label{eq:strongdef}
\begin{array}{ll}
\displaystyle \lim_{\ve \to 0} \left\|u_\ve(t,x) - u_0\left(t, x-\frac{b^*}{\ve}t\right)\right\|_{L^2((0,T)\times \Omega_\ve)} = 0 , \\[0.5cm]
\displaystyle \lim_{\ve \to 0} \sqrt{\ve} \left\|v_\ve(t,x) - f(u_0)\left(t, x-\frac{b^*}{\ve}t\right)\right\|_{L^2((0,T)\times \partial\Omega_\ve)} = 0 .
\end{array}
\end{equation}
Similarly, the gradients of $u_\ve(t,x)$ and $v_\ve(t,x)$ strongly two-scale converge 
with drift in the sense that 
\begin{equation}
\label{eq:strongdef-grad}
\lim_{\ve \to 0} \left\|\nabla u_\ve(t,x) - \nabla u_0\left(t, x-\frac{b^*}{\ve}t\right) - \nabla_y u_1\left(t, x-\frac{b^*}{\ve}t,\frac{x}{\ve}\right) \right\|_{L^2((0,T)\times \Omega_\ve)} = 0 , 
\end{equation}
with $u_1(t,x,y) = \chi(y) \cdot \nabla_x u_0(t,x)$, and
\begin{equation}
\label{eq:strongdef2}
\lim_{\ve \to 0} \sqrt{\ve} \left\| \nabla^s v_\ve(t,x) - G(\frac{x}{\ve}) 
\nabla f(u_0) - \nabla_y^s v_1\left(t, x-\frac{b^*}{\ve}t,\frac{x}{\ve}\right)\right\|_{L^2((0,T)\times \partial\Omega_\ve)} = 0 .
\end{equation}
with $v_1(t,x,y) = \omega(y) \cdot \nabla_x f(u_0)(t,x)$.
\end{thm}

\begin{rem}
If the well prepared assumption (\ref{eq:comp3}) is not satisfied we believe that strong 
convergence, in the sense of Theorem \ref{thm:strong}, still holds true. This was indeed 
proved for the linear case in \cite{AllaireMikelic10}. The mechanism is that, after a 
time $t_0$ as small as we wish, diffusion relaxes any initial data to an almost well prepared 
solution $u_\ve(\cdot,t_0)$, $v_\ve(\cdot,t_0)$ which can serve as a well prepared initial 
data starting at time $t_0$. There are technical difficulties for proving such a result in the 
nonlinear case which we did not overcome. Let us emphasize that the strong convergence of 
the solutions (in the norms of Theorem \ref{thm:strong}) can hold true even though the 
energy associated to the variational formulation of (\ref{eq:p-1})-(\ref{eq:p-3}) does 
not converge to the corresponding homogenized energy (a fact which is well documented, 
see e.g. \cite{BFM}). Note also that we speak of ``energy'' in the mathematical sense 
and they do not seem to have any physical meaning in the context of reactive transport. 
\end{rem}

Before proving this theorem we recall the notion of strong two-scale convergence 
which was originally introduced in Theorem 1.8 in \cite{Allaire92}. It was further 
extended to the case of sequences on periodic surfaces in \cite{Allaire96} and 
to the case of two-scale convergence with drift in \cite{Allaire08}. Of course, 
we can blend these two ingredients and we easily obtain the following result 
that we state without proof. The context and notations are those of Proposition 
\ref{compact-bdry}.

\begin{lem}
\label{prop:strocv}
Let $(V_\ve)_{\ve>0}$ be a sequence in $L^2((0,T)\times\partial\Omega_\ve)$ 
which two-scale converges with drift to a limit $V_0(t,x,y)\in 
L^2((0,T)\times\R^d\times\partial\Sigma^0)$. It satisfies 
\begin{equation}
\label{eq:strocv}
\lim_{\ve\to 0} \sqrt{\ve}\|V_\ve\|_{L^2((0,T)\times\partial\Omega_\ve)} \geq \|V_0\|_{L^2((0,T)\times\R^d\times\partial\Sigma^0)} .
\end{equation}
Assume further that the inequality in (\ref{eq:strocv}) is an equality. 
Then, $V_\ve$ is said to two-scale converges with drift strongly and, 
if $V_0(t,x,y)$ is smooth enough, say $V_0(t,x,y)\in L^2\left( (0,T)\times\R^d ; 
C_\#(\partial\Sigma^0) \right)$, it satisfies 
$$
\lim_{\ve\to0} \ve\dsp\int_0^T \int_{\R^d} \left| V_\ve(t,x) - V_0\left(t,x - \frac{{\mathcal V}}{\ve} t ,\frac{x}{\ve}\right) 
\right|^2 \, {\rm d}x\, {\rm d}t = 0 .
$$
\end{lem}

We now prove a technical result which amounts to say that the $L^2$-norm 
can be replaced by a convex functional in the definition of strong two-scale 
convergence. 

\begin{lem}
\label{lem:convex}
Let ${\mathcal A}:\mathbb{R}\to\mathbb{R}$ be a strongly convex function in the sense 
that there exists a constant $a>0$ such that, for any $u,v\in\mathbb{R}$ and any $\theta\in[0,1]$, 
it satisfies
$$
{\mathcal A}\left(\theta u + (1-\theta) v\right) \leq \theta{\mathcal A}(u) + (1-\theta){\mathcal A}(v) 
-\frac a 2 \theta(1-\theta) \left|u-v\right|^2 .
$$
Let $\{U_\ve(t,x)\}$ be a sequence that two scale converges with drift to $U_0(t,x,y)$. If
$$
\displaystyle \lim_{\ve\to0}\left\|{\mathcal A}(U_\ve)(t,x)\right\|_{L^1((0,T)\times\Omega_\ve)}=\left\|{\mathcal A}(U_0)\left(t, x, y\right)\right\|_{L^1((0,T)\times\mathbb{R}^d\times Y^0)}
$$
then
\begin{equation}
\label{eq.strongcv}
\lim_{\ve \to 0} \left\|U_\ve(t,x) - U_0\left(t, x-\frac{b^*}{\ve}t, \frac{x}{\ve}\right)\right\|_{L^2((0,T)\times \Omega_\ve)} = 0 .
\end{equation}
\end{lem}

\begin{proof}
Since ${\mathcal A}$ is convex and proper (finite), it is continuous and thus, up to an additive 
constant which plays no role, non negative. The strong convexity of ${\mathcal A}$ yields 
$$
{\mathcal A}\left(\theta U_\ve(t,x) + (1-\theta) U_0\left(t,x-\frac{b^*t}{\ve},\frac{x}{\ve}\right)\right)\leq \theta{\mathcal A}(U_\ve)(t,x)
$$
$$
+(1-\theta){\mathcal A}(U_0)\left(t,x-\frac{b^*t}{\ve},\frac{x}{\ve}\right)-\frac a 2 \theta(1-\theta) \left|U_\ve(t,x)-U_0\left(t,x-\frac{b^*t}{\ve},\frac{x}{\ve}\right)\right|^2 .
$$
Taking $\theta=\frac12$ and integrating over $\Omega_\ve\times(0,T)$, we get
$$
\iti\id{\mathcal A}\left(\frac{U_\ve(t,x)+U_0\left(t,x-\frac{b^*t}{\ve},\frac{x}{\ve}\right)}{2}\right)
+ \frac a 8\iti\id\left|U_\ve(t,x)-U_0\left(t,x-\frac{b^*t}{\ve},\frac{x}{\ve}\right)\right|^2
$$ 
\begin{equation}
\label{eq:ineqconvex}
\leq\frac12\iti\id{\mathcal A}(U_\ve)(t,x)+\frac12\iti\id{\mathcal A}(U_0)\left(t,x-\frac{b^*t}{\ve},\frac{x}{\ve}\right) .
\end{equation}
Because of the lower semi-continuity property of convex functions with respect to the weak two-scale convergence with drift, we have
$$
\iti\itr\iy{\mathcal A}(U_0)(t,x,y)\leq\iti\id{\mathcal A}\left(\frac{U_\ve(t,x)+U_0\left(t,x-\frac{b^*t}{\ve},\frac{x}{\ve}\right)}{2}\right) .
$$
Upon passing to the limit, as $\ve\to0$, the right hand side of (\ref{eq:ineqconvex}) 
is exactly equal to the left hand side of the above inequality because of our hypothesis 
on ${\mathcal A}(U_\ve)$, which yields the desired result (\ref{eq.strongcv}).
\end{proof}

We now are ready to prove the main result of this section. 

\begin{proof}[Proof of Theorem \ref{thm:strong}]
Following an idea from \cite{AllaireHutridurga, AllaireMikelic10}, 
we prove that the energy associated with (\ref{eq:p-1})-(\ref{eq:p-3}) converges 
to that of the homogenized equation (\ref{eq:hom}) under assumption (\ref{eq:comp3}). 
Integrating (\ref{eq:energy}) over $(0,t)$ yields
$$
\id F(u_\ve)(t) \, {\rm d}x + \frac{\ve}{2} \ib |v_\ve(t)|^2 \, {\rm d}\sigma(x) + \itt \id f'(u_\ve) D_\ve \nabla u_\ve \cdot \nabla u_\ve \, {\rm d}x \, {\rm d}s
$$
$$
+ \ve\itt \ib D^s_\ve\nabla^s v_\ve \cdot\nabla^s v_\ve \, {\rm d}\sigma(x)\, {\rm d}s + \kappa\ve \itt\ib \left(w_\ve\right)^2 \, {\rm d}\sigma(x)\, {\rm d}s
$$
$$
= \id F(u^{in}) \, {\rm d}x + \frac{\ve}{2} \ib |v^{in}|^2 \, {\rm d}\sigma(x) ,
$$
with $w_\ve = (f(u_\ve) - v_\ve)/\ve$. 
Since two-scale convergence with drift holds only in a time-space product interval, 
we integrate again the above expression over $(0,T)$ to get
$$
\iti\id F(u_\ve)(t) \, {\rm d}x\, {\rm d}t + \frac{\ve}{2} \iti\ib |v_\ve(t)|^2 \, {\rm d}\sigma(x)\, {\rm d}t + 
\iti\itt \id f'(u_\ve) D_\ve \nabla u_\ve \cdot \nabla u_\ve \, {\rm d}x \, {\rm d}s\, {\rm d}t
$$
$$
+ \ve\iti\itt \ib D^s_\ve\nabla^s v_\ve\cdot\nabla^s v_\ve \, {\rm d}\sigma(x)\, {\rm d}s\, {\rm d}t
+ \kappa\ve \iti\itt\ib \left(w_\ve\right)^2 \, {\rm d}\sigma(x)\, {\rm d}s\, {\rm d}t
$$
$$
= T \id F(u^{in}) \, {\rm d}x + \frac{T\ve}{2} \ib |v^{in}|^2 \, {\rm d}\sigma(x) .
$$
We pass to the two-scale limit in all terms of the above left hand side 
by using the lower semi-continuity property of norms and of the convex 
function $F$. The only delicate term is the third one, involving the 
nonlinear term $f'(u_\ve)$, where we use 
again the compactness of Corollary \ref{cor:comp-u}. 
Passing to the limit yields the following inequality
$$
|Y^0|\iti\itr F(u_0) \, {\rm d}x \, {\rm d}t + \frac 12|\partial\Sigma^0|\iti\itr (f(u_0))^2\, {\rm d}x\, {\rm d}t
$$
$$
+ \kappa\int\limits_0^T\int\limits_0^t\itr\ip |\left(\chi(y) - \omega(y)\right)\cdot \nabla_x f(u_0)(s,x)|^2\, {\rm d}\sigma(y) \, {\rm d}x \, {\rm d}s \, {\rm d}t.
$$
$$
+ \iti\itt\itr\iy f'(u_0) D(y) |\nabla_x u_0(s,x) + \nabla_y\left(\chi(y)\cdot \nabla_x u_0(s,x)\right)|^2\, {\rm d}y \, {\rm d}x \, {\rm d}s \, {\rm d}t.
$$
$$
+ \iti\itt\itr\ip D^s(y) |G(y) \nabla_x f(u_0)(s,x) + \nabla^s_y\left(\omega(y)\cdot \nabla_x f(u_0)(s,x)\right)|^2 \, {\rm d}\sigma(y) \, {\rm d}x \, {\rm d}s \, {\rm d}t 
$$
$$
\leq T |Y^0|\itr F(u^{in}) \, {\rm d}x + \frac{T}{2} |\partial\Sigma^0|\itr |v^{in}|^2 \, {\rm d}x .
$$
Recognizing formula (\ref{exp-disp}) for $A^*$ leads to
$$
|Y^0|\iti\itr F(u_0) \, {\rm d}x \, {\rm d}t + \frac 12|\partial\Sigma^0|\iti\itr (f(u_0))^2\, {\rm d}x\, {\rm d}t
$$
$$
+ \iti\itt\itr A^*(u_0)\nx u_0 \cdot\nx(f(u_0) )\, {\rm d}x\, {\rm d}s\, {\rm d}t
$$
\begin{equation}
\leq T |Y^0|\itr F(u^{in}) \, {\rm d}x + \frac{T}{2} |\partial\Sigma^0|\itr |v^{in}|^2 \, {\rm d}x .
\label{eq:rhsineq} 
\end{equation}
We now compare inequality (\ref{eq:rhsineq}) with the (time integral of the) energy equality 
for the homogenized equation (\ref{eq:hom}) with $f(u_0)$ as a test function
$$
|Y^0|\iti\itr F(u_0) \, {\rm d}x \, {\rm d}t + \frac 12|\partial\Sigma^0|\iti\itr (f(u_0))^2\, {\rm d}x\, {\rm d}t
$$
$$
+ \iti\itt\itr A^*(u_0)\nx u_0 \cdot\nx(f(u_0) )\, {\rm d}x\, {\rm d}s\, {\rm d}t
$$
\begin{equation}
= T |Y^0|\itr F(u_0)(0,x)\, {\rm d}x + \frac T2 |\partial\Sigma^0| \itr (f(u_0))^2(0,x)\, {\rm d}x .
\label{eq:rhseq} 
\end{equation}
The right hand side in (\ref{eq:rhsineq}) and (\ref{eq:rhseq}) are equal precisely 
when (\ref{eq:comp1}) holds true. Together with the definition (\ref{eq:comp2}) of 
the initial condition of the homogenized problem (\ref{eq:hom}), it is equivalent 
to our assumption (\ref{eq:comp3}). 
In such a case, the inequality (\ref{eq:rhsineq}) is actually an equality, meaning 
that the lower semi continuous convergences leading to (\ref{eq:rhsineq}) were 
exact convergences. Then, applying Lemma \ref{lem:convex} and Lemma \ref{prop:strocv} 
gives the result (\ref{eq:strongdef}).
\end{proof}

\begin{lem}
\label{lem:comp}
For any given $u^{in}\geq0$ there always exists a unique solution $v^{in}\geq0$ 
of the nonlinear equation (\ref{eq:comp3}), ${\mathcal H}(u^{in},v^{in})=0$. 
\end{lem}

\begin{proof}
Define $\eta = |\partial\Sigma^0| / |Y^0|$. 
Since $0\leq f(u)\leq \alpha u$, the function $u \to  u + \eta f(u)$ 
is monotone and invertible on $\mathbb{R}^+$. Therefore, (\ref{eq:comp2}) uniquely 
defines the homogenized initial data as
\begin{equation}
\label{eq:comp4}
u_0^0 = \left( I + \eta f \right)^{-1} 
\left( u^{in} + \eta v^{in} \right) . 
\end{equation}
To satisfy the additional relation (\ref{eq:comp1}) is equivalent to solving the 
nonlinear equation (\ref{eq:comp3}) where ${\mathcal H}$ is defined by 
\begin{equation}
\label{eq:comp5}
{\mathcal H}(u^{in},v^{in})= 
F(u^{in}) + \frac12 \eta (v^{in})^2 - 
\left( F + \frac 12 \eta f^2 \right) (u^0_0) ,
\end{equation}
where $u^0_0$ is defined by (\ref{eq:comp4}). For a given $u^{in}\geq0$, 
let us differentiate ${\mathcal H}$ with respect to $v^{in}$:
\begin{equation}
\label{eq:DH}
\partial_{v^{in}} {\mathcal H}(u^{in},v^{in}) = \eta v^{in} - (f(u^0_0) + \eta f(u^0_0) f'(u^0_0))\partial_{v^{in}}u^0_0 .
\end{equation}
Differentiating (\ref{eq:comp2}) with respect to $v^{in}$ leads to
\begin{equation}
\label{eq:D-hom-init}
(1 + \eta f'(u^0_0))\partial_{v^{in}}u^0_0 = \eta ,
\end{equation}
implying that $\partial_{v^{in}}u^0_0>0$. Using (\ref{eq:D-hom-init}) in (\ref{eq:DH}) simplifies the derivative of ${\mathcal H}$ with respect to $v^{in}$ as
\begin{equation}
\label{eq:DH-final}
\partial_{v^{in}} {\mathcal H}(u^{in},v^{in}) = \eta v^{in} - \eta f(u^0_0) .
\end{equation}
Since $f\geq0$, we have
$$
\partial_{v^{in}} {\mathcal H}(u^{in},0) = - \eta f(u^0_0) \le0.
$$
Also since $f\le \alpha/\beta$, we have
$$
\dsp\lim_{v^{in}\to+\infty}\partial_{v^{in}}{\mathcal H}(u^{in},v^{in}) = +\infty.
$$
Let us differentiate (\ref{eq:DH-final}) with respect to $v^{in}$:
\begin{equation}
\label{eq:DDH}
\partial^2_{v^{in}} {\mathcal H}(u^{in},v^{in}) = \eta - \eta f'(u^0_0)\partial_{v^{in}}u^0_0 .
\end{equation}
Using (\ref{eq:D-hom-init}) in (\ref{eq:DDH}) leads to
$$
\partial^2_{v^{in}} {\mathcal H}(u^{in},v^{in}) = \partial_{v^{in}}u^0_0 > 0 .
$$
Thus, for a fixed $u^{in}$, $v^{in}\to\partial_{v^{in}}{\mathcal H}(u^{in},v^{in})$ 
is continuous monotone increasing function so there exists a unique $v^{in}_*$ such that 
$\partial_{v^{in}}{\mathcal H}(u^{in},v^{in}_*)=0$. By (\ref{eq:DH-final}) we have 
$v^{in}_* = f(u^0_0)$ and from (\ref{eq:comp2}) we deduce $u^{in}=u^0_0$. 
Plugging these values in (\ref{eq:comp4}) implies that ${\mathcal H}(u^{in},v^{in}_*) = 0$. 
Since the function $v^{in} \to {\mathcal H}(u^{in},v^{in})$ is decreasing from 0 to 
$v^{in}_*$ and then increasing, $v^{in}_*$ is the only possible root for ${\mathcal H}$. 
\end{proof}

\section{Numerical Study}
\label{sec:num}

From a physical or engineering point of view, one of the main consequences of 
our homogenization result it to provide a formula to compute the so-called 
dispersion tensor $A^*$ which governs the spreading of the solute at a macroscopic 
scale. Formula (\ref{exp-disp}) for $A^*$ is not fully explicit with respect 
to the various physical parameters. Therefore it is interesting to study the 
sensitivity of $A^*$ with respect to important parameters like the concentration 
saturation $u_0$, or the bulk and surface diffusion $D$, $D^s$. 
This section is thus devoted to numerical computation of the effective parameters, 
given in Theorem \ref{main-weak}, and to study their variations in terms of these 
parameters. All our numerical tests are done in two dimensions. The periodicity cell 
is the unit square $]0,1[\times]0,1[$ and the solid obstacle is a disk of radius $0.2$ 
centered at $(0.5,0.5)$. The FreeFem++ package \cite{Pirofreefem} is used to perform 
all numerical simulations with Lagrange P1 finite elements on $21416$ vertices 
(degrees of freedom). The reaction parameters $\alpha$, $\beta$ are chosen to be unity. 
Also, the bulk and surface diffusion $D$, $D^s$ are taken to be unity.

In all our computations, we have taken a zero drift velocity:
\begin{equation}
\label{eq:drift0}
b^* = 0 .
\end{equation}
This is actually a necessary condition for the present geometrical setting of 
isolated solid obstacles. Indeed, recall our assumption (\ref{eq:drift}) on the 
velocity fields $b$, $b^s$:
$$
b^* = \frac{1}{|Y^0|}\int_{Y^0}b(y)\, {\rm d}y = \frac{1}{|\partial\Sigma^0|}\int_{\partial\Sigma^0}b^s(y)\, {\rm d}\sigma(y) .
$$
Since the surface velocity field $b^s$ is divergence free and the obstacle 
$\Sigma^0$ is compactly included in the unit cell $Y^0$ (therefore the manifold 
$\partial\Sigma^0\cap Y^0$ has no boundary), an integration by parts 
shows that
$$
\int_{\partial\Sigma^0} e_k \cdot b^s(y)\, {\rm d}\sigma(y) = - 
\int_{\partial\Sigma^0} y_k \div^s b^s(y)\, {\rm d}\sigma(y) = 0 .
$$
(It is only in the case of a connected solid part, which can happen only in 
dimension $d\geq3$, that one can have $b^*\neq0$.) 
For simplicity we have taken $b^s=0$ on $\partial\Sigma^0$. 
We have computed the mean zero velocity field $b$ in $Y^0$ by taking $\tilde b=\textrm{curl}\:\psi=(-\partial_{x_2}\psi,\partial_{x_1}\psi)$, with
\begin{equation}
\left\{
\begin{array}{ll}
 -\div({\mathcal M}(y)\nabla\psi) = 1 & \: \: \textrm{in} \: \: Y^0,\\
\psi = 0 & \: \: \textrm{on} \: \: \partial\Sigma^0,\\
\psi & \:\: Y-\textrm{periodic,}
\end{array}\right.
\label{eq:velocity-curl}
\end{equation}
where $\mathcal{M}$ is a $2\times2$ matrix and $b=\tilde b/\|\tilde b\|_{L^2(Y^0)}$ shall be the normalized velocity field. We can choose the velocity field $b$ to be either symmetric or non-symmetric. By symmetry we mean the symmetric nature of the velocity field in the unit cell $Y$ with respect to axes $(0.5,y)$ and $(x,0.5)$. For example, taking $\mathcal M$ to be an identity matrix we obatin a symmetric velocity field. On the other hand, taking $\mathcal M$ to be a variable diagonal matrix with the following diagonal elements we obtain a non-symmetric velocity field:
$$
{\mathcal M}_{11} =\left\{\begin{array}{ll}
                      0.01 + (0.5*y_1) & \textrm{if}\:\:y_1 < 0.5\\
                      0.26 + (y_1 - 0.5) & \textrm{otherwise}
                      \end{array}\right.,
$$
$$
{\mathcal M}_{22} = \cos(y_1).
$$
Thus obtained velocity fields are shown in Figure \ref{fig:1}.
\begin{figure}[!ht]
\begin{center}
\includegraphics[width=6.0 cm]{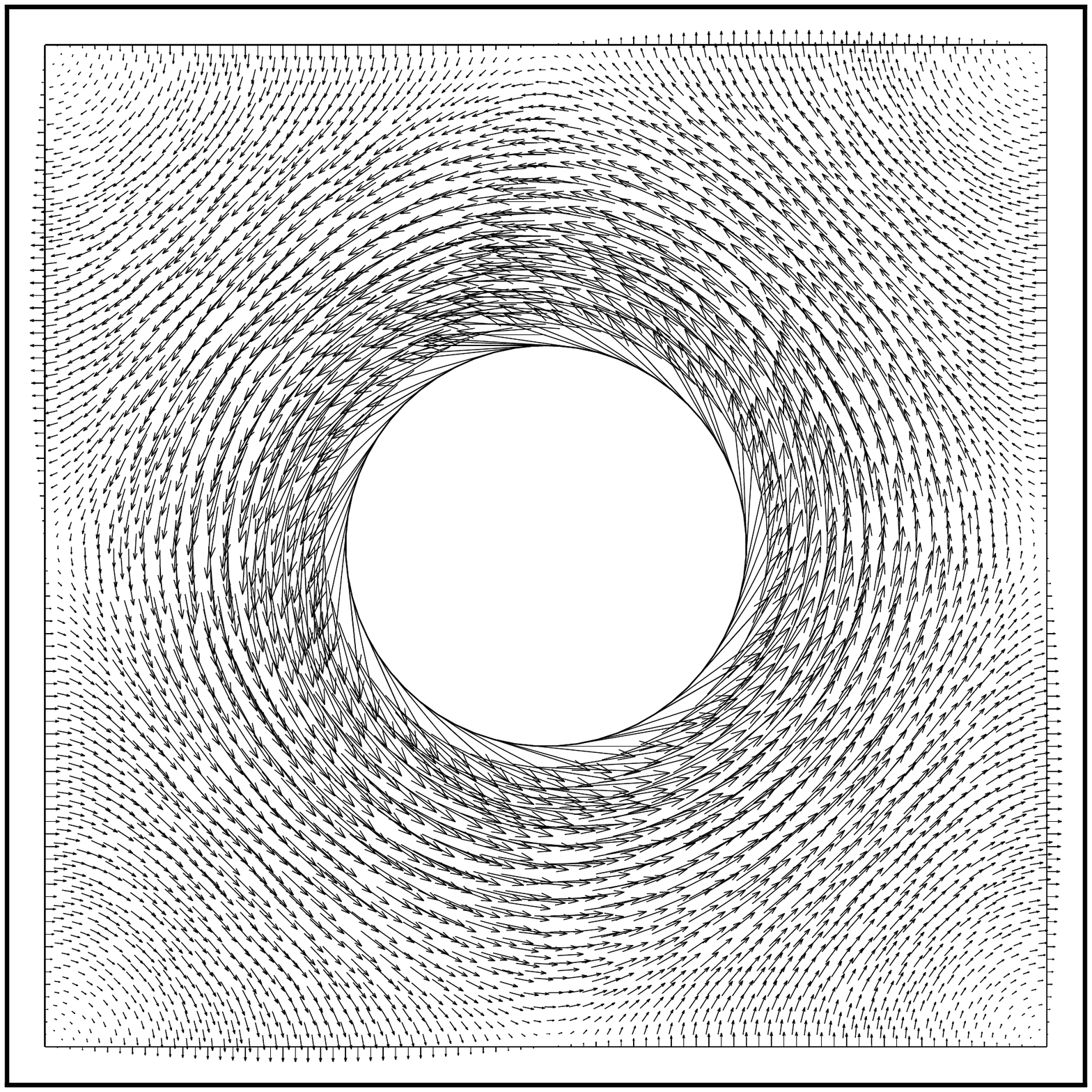}
\includegraphics[width=6.0 cm]{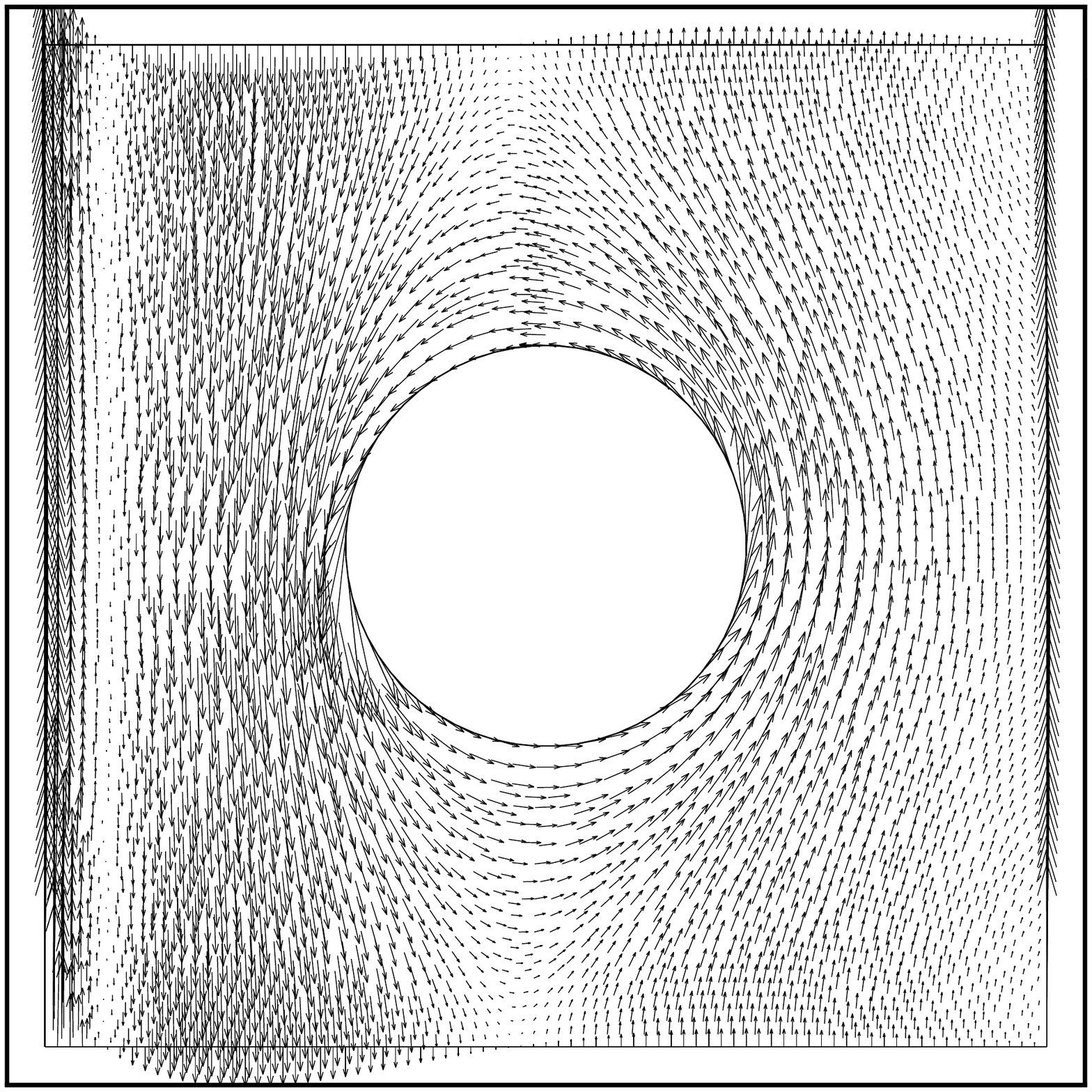}
\caption{Left: Symmetric velocity field in the unit cell, Right: Non-symmetric velocity field in the unit cell.}
\label{fig:1}
\end{center}
\end{figure}

The expression for the dispersion matrix given in (\ref{exp-disp}) implies that the dispersion tensor depends on the homogenized solution. In a first experiment, we study the behavior of $A^*_{11}$ and $A^*_{22}$ with respect to the magnitude of $u_0$ when the velocity field $b$ is symmetric (See Figure \ref{fig:2a}). In our second experiment, we take the velocity field $b$ to be non-symmetric and study the behavior of $A^*_{11}$ and $A^*_{22}$ with respect to the magnitude of $u_0$ (See Figure \ref{fig:2b}). As seen in Figures \ref{fig:2a} and \ref{fig:2b}, in the limit $u_0\to\infty$, both horizontal and vertical dispersion attain a limit. It is easy to see, at least formally, that in this limit, the cell problem is partially decoupled: the bulk cell solution $\chi_i$ satisfies the following steady state equation:
\begin{equation}
\label{formal-limit-u0-infin}
\left\{ 
\begin{array}{ll}
-b^* \cdot e_i + b(y)\cdot(e_i + \ny \chi_i) - \div_y(D(e_i + \ny \chi_i)) = 0 & \textrm{in } Y^0,\\[0.3cm]
- D\left( e_i + \ny \chi_i\right)\cdot n = 0 & \textrm{on } \partial \Sigma^0, \\[0.3cm]
y \to \chi_i(y) \quad  Y\mbox{-periodic,} &
 \end{array}\right.
\end{equation}
while the surface cell solution $\omega_i$ satisfies another equation where $\chi_i$ acts as a source term:
$$
-b^* \cdot e_i + \kappa\omega_i + b^s(y)\cdot(e_i + \ny^s \omega_i) - \div^s_y(D^s(e_i+ \nabla^s_y \omega_i)) = \kappa\chi_i\:\:\mbox{ on }\partial\Sigma^0.
$$
In our case, we have taken the velocity field $b^s$ to be zero and also the drift $b^*$ is zero. Thus, the above equation for $\omega_i$ is a simple elliptic equation with a source term.

\begin{figure}[!ht]
\begin{center}
\includegraphics[width=7.0 cm]{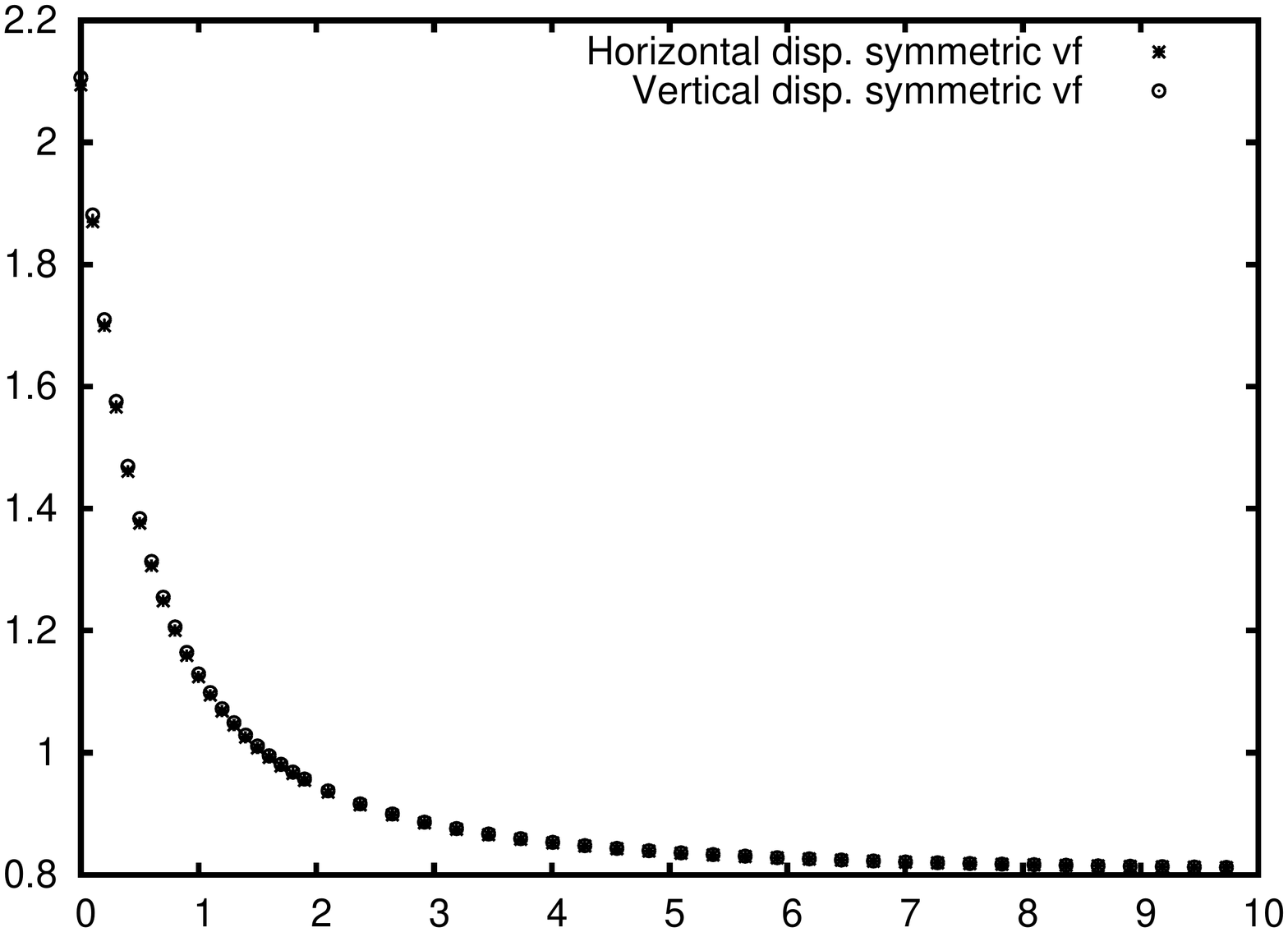}
\caption{Dispersion w.r.t the magnitude of $u_0$ in the case of a symmetric velocity field.}
\label{fig:2a}
\end{center}
\end{figure}

\begin{figure}[!ht]
\begin{center}
\includegraphics[width=7.0 cm]{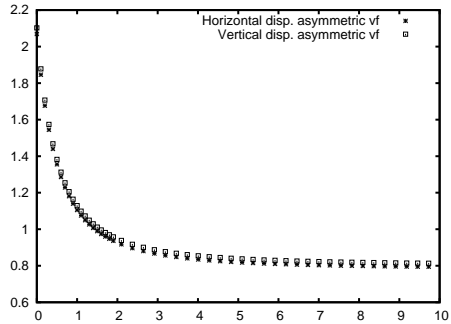}
\caption{Dispersion w.r.t the magnitude of $u_0$ in the case of a non-symmetric velocity field.}
\label{fig:2b}
\end{center}
\end{figure}

In Figure \ref{fig:3}, we plot the horizontal dispersion $A^*_{11}$ with respect to $D^s$ with $u_0=2.5$. Clearly the dispersion increases with the surface diffusion $D^s$. However, as seen in Figure \ref{fig:3}, the dispersion reaches a limit as $D^s$ goes to infinity. This can be explained formally by the fact that, in this limit, the surface cell solution $\omega_i$ is such that $(\omega_i + y_i)$ is constant on the pore surface $\partial\Sigma^0$. In the same limit, the bulk corrector $\chi_i$ satisfy the following limit problem:
\begin{equation}
\label{formal-limit-ds-2}
\left\{
\begin{array}{ll}
b(y) \cdot \ny \chi_i - \div_y (D (\ny \chi_i + e_i)) = (b^* - b) \cdot e_i & \textrm{in} \: \: Y^0,\\[0.3 cm]
-D (\ny\chi_i +  e_i )\cdot n +  b^*_i = \\ 
\dsp\frac{\alpha \kappa}{(1+\beta u_0)^2} \left(\chi_i + y_i - |\partial\Sigma^0|^{-1} \dsp\int_{\partial\Sigma^0} 
(\chi_i + y_i )d\sigma(y) \right) & \textrm{on} \: \: \partial\Sigma^0,\\
y \to \chi_i(y) & Y\mbox{-periodic.}
\end{array} \right.
\end{equation}

\begin{figure}[!ht]
\begin{center}
\includegraphics[width=7.0 cm]{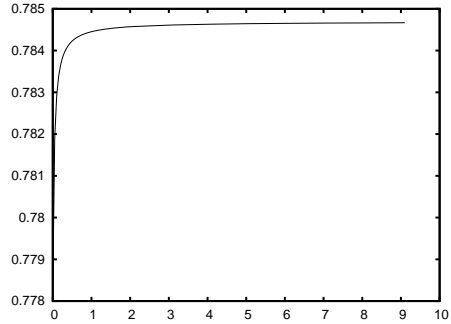}
\caption{Horizontal Dispersion w.r.t surface molecular diffusion $D^s$.}
\label{fig:3}
\end{center}
\end{figure}

In Figure \ref{fig:10}, we plot the horizontal dispersion $A^*_{11}$ with respect to the reaction rate $\kappa$ with $u_0=2.5$. In the limit $\kappa \to \infty$, we get an asymptote for the dispersion, corresponding to a limit cell problem where $\chi_i=\omega_i$ on $\partial\Sigma^0$. In this limit, the corresponding system satisfied by the bulk corrector $\chi_i$ is
\begin{equation}
\label{formal-limit-kappa}
\left\{
\begin{array}{ll}
b(y) \cdot \ny \chi_i - \div_y (D (\ny \chi_i + e_i)) = (b^* - b) \cdot e_i & \textrm{in} \: \: Y^0,\\[0.3 cm]
- D (\ny\chi_i +  e_i )\cdot n + ( b^* - b^s)\cdot e_i = & \\
\hspace{2cm} b^s(y) \cdot \ny^s \chi_i - \div^s_y (D^s (\ny^s \chi_i + e_i)) & \textrm{on} \: \: \partial\Sigma^0,\\
y \to \chi_i(y) & Y\mbox{-periodic.}
\end{array} \right.
\end{equation}
Unlike (\ref{formal-limit-ds-2}), the limit cell problem corresponding to the infinite reaction limit is no longer dependent on the homogenized solution $u_0$.
\begin{figure}[!ht]
\begin{center}
\includegraphics[width=7.0 cm]{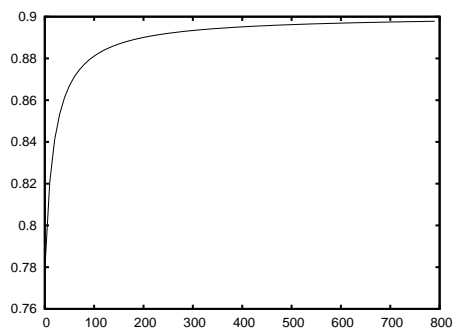}
\caption{Horizontal Dispersion w.r.t reaction rate $\kappa$.}
\label{fig:10}
\end{center}
\end{figure}

Other numerical simulations, including comparisons between an ``exact'' solution of (\ref{eq:p-1})-(\ref{eq:p-3}) (computed on a fine mesh) and a reconstructed solution can be found in \cite{Harsha}.

\signga

\signhh

\end{document}